\documentclass[11pt]{amsart}
\usepackage[T1]{fontenc}
\usepackage[ansinew]{inputenc}
\usepackage{amssymb,amsmath,amsthm,eucal,hyperref,mathrsfs,array,graphicx,xcolor,bbm,enumerate,wasysym}
\usepackage[mathlines,displaymath,pagewise]{lineno}
\usepackage[all]{xy}
\usepackage{fullpage}

\numberwithin{equation}{section}

\newtheorem{lemma}{Lemma}[section]
\newtheorem{propn}[lemma]{Proposition}
\newtheorem{thm}[lemma]{Theorem}
\newtheorem{cor}[lemma]{Corollary}
\newtheorem{defn}[lemma]{Definition}
\newtheorem{remark0}[lemma]{Remark}

\DeclareMathOperator{\CR}{CR}

\renewenvironment{proof}{{\em Proof.}}{\hspace*{\fill} $\square$}
\newenvironment{proofof}[1]{{\em Proof of #1.}}{\hspace*{\fill} $\square$}

 \newcommand{\R}{\mathbb{R}}
 \newcommand{\C}{\mathbb{C}}
 \newcommand{\N}{\mathbb{N}}
 
 \newcommand{\I}{\mathbbm{1}}
 \newcommand{\E}[1]{\mathbb{E}\left [ #1 \right ]}
 \newcommand{\Ea}[1]{\mathbb{E}^{A_n}\left [ #1 \right ]}
 \newcommand{\HH}{\mathbb{H}}
 \newcommand{\Os}{\mathcal{O}}
 \newcommand{\F}{\mathcal F}
 
 \newcommand{\e}{\operatorname{e}}
 
 \newcommand{\D}{\mathbb D }
 \renewcommand{\I}[1]{\mathbf 1_{\{#1\}}}
  \renewcommand{\1}{\mathbf 1}
  \renewcommand{\P}{\mathbb P}
 
 \newcommand{\dne}{\tilde{D}_n^\eta}
 
 \newcommand{\ind}[2]{\1_{E_\eta(#1,#2)}}
 \newcommand{\qex}{\mathbb{Q}_{\eta,z}^*}
 \newcommand{\qa}[1]{\mathbb{Q}_{\eta,z}^*\left[#1\right]}
 \newcommand{\qb}[1]{\mathbb{Q}_\eta^*\left[#1\right]}

 \newcommand{\eps}{\varepsilon}
 
 \newcommand{\lzn}{\log\CR^{-1}(z,\D\setminus A_n)}
 \newcommand{\lzo}{\log\CR^{-1}(z,\D)}
 \newcommand{\Ev}{\mathtt{E}}

\begin{document}

\title{Liouville measure as a multiplicative cascade via level sets of the Gaussian free field} 
\author{Juhan Aru, Ellen Powell, Avelio Sep\'{u}lveda}
\date{}
\begin{abstract}
We provide new constructions of the subcritical and critical Gaussian multiplicative chaos (GMC) measures corresponding to the 2D Gaussian free field (GFF). As a special case we recover E. Aidekon's construction of random measures using nested conformally invariant loop ensembles, and thereby prove his conjecture that certain CLE$_4$ based limiting measures are equal in law to the GMC measures for the GFF. The constructions are based on the theory of local sets of the GFF and build a strong link between multiplicative cascades and GMC measures. This link allows us to directly adapt techniques used for multiplicative cascades to the study of GMC measures of the GFF. As a proof of principle we do this for the so-called Seneta--Heyde rescaling of the critical GMC measure. 
\end{abstract}
\maketitle

\section{Introduction} 

Gaussian multiplicative chaos (GMC) theory, initiated by Kahane in the 80s \cite{KAH} as a generalization of multiplicative cascades, aims to give a meaning to ``$\exp(\Gamma)$'' for rough Gaussian fields $\Gamma$. In a simpler setting it was already used in the 70s to model the exponential interaction of bosonic fields \cite{HK}, and over the past ten years it has gained importance as a key component in constructing probabilistic models of so-called Liouville quantum gravity in 2D \cite{DS, DKRV} (see also \cite{Nakayama} for a review from the perspective of theoretical physics).

One of the important cases of GMC theory is when the underlying Gaussian field is equal to $\gamma\Gamma$, for $\Gamma$ a 2D Gaussian free field (GFF) \cite{DS} and $\gamma > 0$ a parameter. It is then possible to define random measures with area element ``$\exp(\gamma \Gamma)dx\wedge dy$''. These measures are sometimes also called Liouville measures \cite{DS}, and we will do so for convenience in this article. \footnote{In the physics literature, the term ``Liouville measure" refers to a volume form coming from a conformal field theory with a non-zero interaction term (see \cite[Section 3.6]{RVnotes} for an explanation). This induces a certain weight on the underlying GFF measure. Therefore, the {G}aussian multiplicative chaos measures that we consider in this article are not precisely the Liouville measures from the physics literature. Our measures correspond in some sense only to a degenerate case, where the interaction parameter is set to $0$.} Due to the recent work of many authors \cite{RV,DS, Ber, Shamov} one can say that we have a rather thorough understanding of Liouville measures in the so-called subcritical regime ($\gamma < 2$). The critical regime ($\gamma = 2$) is trickier, but several constructions are also known \cite{DSRV, DSRV2, JS,EP}. 

Usually, in order to construct the GMC measure, one first approximates the underlying field using either a truncated series expansion or smooth mollifiers, then takes the exponential of the approximated Gaussian field, renormalizes it and shows that the limit exists in the space of measures. In a beautiful paper \cite{Ai} the author proposed a different way to construct measures of multiplicative nature using nested conformally invariant loop ensembles, inspired by multiplicative cascades. He conjectured that in the subcritical and critical regime, and in the case where these loop ensembles correspond to certain same-height contour lines of the underlying GFF, the limiting measure should have the law of the Liouville measure. In this paper we confirm his conjecture. This is done by providing new constructions of the subcritical and critical Liouville measures using a certain family of so called local sets of the GFF \cite{SchSh2, ASW} and reinterpreting his construction as a special case of this general setting. Some of our local-set based constructions correspond to simple multiplicative cascades, and others in some sense to stopping lines constructions of the multiplicative cascade measures \cite{K}. Moreover, although the underlying field is Gaussian, our approximations are ``non-Gaussian'' but yet both local and conformally invariant. Note that for the 1D chaos measures there are recent non-Gaussian constructions stemming from random matrix theory, see e.g. \cite{Webb}, but they are very different in nature. We also remark that our construction strongly uses the Markov property of the GFF and hence does not easily generalize to other log-correlated fields.

One simple, but important, consequence of our results is the simultaneous construction of a GFF in a simply connected domain and its associated Liouville measure using nested CLE$_4$ and a collection of independent coin tosses. Start with a height function $h_0 = 0$ on $\D$ and sample a CLE$_4$ in $\D$. Inside each connected component of its complement add either $\pm \pi$ to $h_0$ using independent fair coins. Call the resulting function $h_1$. Now repeat this procedure independently in each connected component: sample an independent CLE$_4$, toss coins and add $\pm \pi$ to $h_1$ to obtain $h_2$. Iterate. Then it is known \cite{MS,ASW} that these piecewise constant fields $h_n$ converge to a GFF $\Gamma$. It is also possible to show that the nested CLE$_4$ used in this construction is a measurable function of $\Gamma$. Proposition \ref{propn::mainresult_withF} of the current article implies that one can construct the Liouville measures associated to $\Gamma$ by just taking the limit of measures \[M_n^\gamma(dz) = e^{\gamma h_n(z)}\CR(z; \D\setminus A^n)^\frac{\gamma^2}{2}dz.\]
Here $\CR(z; \D\setminus A^n)$ is the conformal radius of the point $z$ inside the $n$-th level loop. 

Observe that the above approximation is different from taking naively the exponential of $h_n$ and normalizing it pointwise by its expectation. In fact, it is not hard to see that in this setting the latter naive procedure that is used for mollifier and truncated series approximations would not give the Liouville measure. 

In the critical case, and keeping to the above concrete approximation of the GFF,  regularized Liouville measures can be given by the so-called derivative approximations:
\[D_n(dz) =  \left(-h_n(z) + 2 \log \CR^{-1}(z,\D\setminus A^n)\right)e^{2 h_n(z)}\CR(z; \D\setminus A^n)^{2}\, dz.\]
As the name suggests, they correspond to (minus) the derivative of the above measure $M_n^\gamma$ w.r.t. to $\gamma$, taken at the critical parameter $\gamma = 2$. We show that these approximate signed measures converge to a positive measure that agrees (up to a constant factor 2) with the limiting measure of \cite{Ai} described in Section \ref{sectLM}, and also to the critical Liouville measure constructed in \cite{DSRV2,EP}.

The connection between multiplicative cascades and the Liouville measure established by our construction makes it possible to directly adapt many techniques developed in the realm of branching random walks and multiplicative cascades to the study of the Liouville measure. This allows us to prove a ``Seneta--Heyde'' rescaling result in the critical regime by following closely the arguments for the branching random walk in \cite{AiSh}. In a follow-up paper, we will use this result to transfer another result from cascades, \cite{Madaule}, to the case of the Liouville measure, and to thereby answer a conjecture of \cite{DSRV} in the case of the GFF: we prove that under a suitable scaling, the subcritical measures converge to a multiple of the critical measure. Finally, our proofs are robust enough to study the Liouville measure in non-simply connected domains and also to study the boundary Liouville measure.

The rest of the article is structured as follows. We start with preliminaries on the GFF, its local sets and Liouville measure. Then, we treat the subcritical regime and discuss generalizations to non-simply connected domains and to the boundary Liouville measure. Finally, we handle the critical case: we first show that our construction agrees with both a construction by E. Aidekon (up to a constant factor $2$) and a mollifier construction of the critical Liouville measure; then, we consider the case of Seneta-Heyde scaling.

\section{Preliminaries on the Gaussian free field and its local sets}
\label{sectGFF}

Let $D\subseteq \R^2$ denote a bounded, open and simply connected planar domain. By conformal invariance, we can always assume that $D$ is equal to $\D$, the unit disk. Recall that the Gaussian Free Field (GFF) in $D$  can be viewed as a centered Gaussian process $\Gamma$, 
indexed by the set of continuous functions in $D$, with covariance given by 
\begin{equation}\label{GFF}
\E {(\Gamma,f) (\Gamma,g)}  =  \iint_{D\times D} f(x) G_D(x,y) g(y) d x d y.
\end{equation} 
Here $G_D$ is the Dirichlet Green's function  in $D$, normalized such that $G_D(x,y)\sim \log(1/|x-y|)$ as $x \to y$ for all $y \in D$. 

Let us denote by $\rho^\eps_z$ the uniform measure on the circle of radius $\eps$ around $z$. Then for all $z \in D$ and all $\eps > 0$, one can define $\Gamma_\eps := (\Gamma,\rho_z^\eps)$.  We remark that this concrete choice of mollifying the free field is of no real importance, but is just a bit more convenient in the write-up of the critical case.
 
An explicit calculation (see e.g. Proposition 3.2. in \cite{DS}) then shows that:
\begin{equation}\label{eqn::condmeanGammaA}
\E{\ \eps^{\frac{\gamma^2}{2}} \exp\left(\gamma (\Gamma,\rho^\eps_z)\right)}\left\{ \begin{array}{l l}
= \CR(z;D)^{\gamma^2/2} & \text{ if } d(z,\partial D)\geq \eps,\\
\leq 1 & \text{ if } d(z,\partial D)< \eps,
\end{array}\right.
\end{equation} 
where  $\CR(z;D)$ is the conformal radius of $z$ in the simply-connected domain $D$. This exact formula holds not only for the circle average, but for any mollifier $\tilde \rho_z^\eps$ that is radially-symmetric and supported in the disk of radius $\eps$ around $z$.

The Gaussian free field satisfies a spatial Markov property, and in fact it also satisfies a strong spatial Markov property. To formalise this, the concept of local sets was introduced in \cite{SchSh2}. They can be thought as the generalisation of stopping times to a higher dimension. 

\begin{defn}[Local sets]
	Consider a random triple $(\Gamma, A,\Gamma_A)$, where $\Gamma$ is a  GFF in $D$, $A$ is a random closed subset of $\overline D$  and $\Gamma_A$ a random distribution that can be viewed as a harmonic function, $h_A$, when restricted to $D \setminus A$.
	We say that $A$ is a local set for $\Gamma$ if conditionally on $A$ and $\Gamma_A$, $\Gamma^A:=\Gamma - \Gamma_A$ is a  GFF in $D \setminus A$. 
\end{defn}

Here, by a random closed set we mean a probability measure on the space of closed subsets of $\overline D$, endowed with the Hausdorff metric and its corresponding Borel $\sigma-$algebra. For simplicity, we will only work with local sets $A$ that are measurable functions of $\Gamma$ and such that $A\cup \partial D$ is connected. In particular, this implies that all connected components of $D\backslash A$ are simply-connected. We define $\F_A=\sigma(A)\vee \sigma (\Gamma_A)$. 

Other than the Markov property apparent from the definition, we will use the following simple properties of local sets. See for instance \cite {SchSh2,WWln2} for  further properties. 
	\begin{lemma}\label{BPLS}    $\ $
Let  $(A^n)_{n\in \N}$  be an increasing sequence of local sets measurable w.r.t. $\Gamma$. Then 
		\begin {enumerate} 
		\item $\F_{A^n} \subset \F_{A^{n+1}},$
		\item $\overline{\bigcup  A^n}$ is also a local set and $\Gamma_{A_N}\to\Gamma_{\overline{\bigcup  A^n}} $ in probability as $N\to \infty,$
		\item if $\overline{\bigcup  A^n}= \overline D$, then the join of the $\sigma$-algebras $\F_{A^n}$ is equal to $\sigma(\Gamma)$. Moreover, $\Gamma_n := \Gamma_{A^n}$ then converges to $\Gamma$ in probability in the space of distributions.
		\end {enumerate}
	\end{lemma} 
The property (1) follows from the fact that our local sets are measurable w.r.t. $\Gamma$ and the characterization of local sets found in \cite{SchSh2}. Properties (2) and (3) follow from the fact that when $A^n\cup \partial D$ is connected we have that $G_{D\backslash A^n}\to G_{D\backslash \overline{\bigcup A^n}}$.

In other words, one can approximate the Gaussian free field by taking an increasing sequence of measurable local sets $(A^n)_{n\in \N}$ and for each $n$ defining {$\Gamma_n := \Gamma_{A^n}$. As $\Gamma_n$ are measurable w.r.t. the GFF and also piecewise harmonic, they give very simple intrinsic approximations to the field. For example, one could intuitively think that $A^n$ are the sets that discover the part of the surface described by the GFF that is linked to the boundary and on which the GFF has height between $-n$ and $n$.

\subsection{Two useful families of local sets}

One useful family of local sets are the so-called two-valued local sets introduced in \cite{ASW} and denoted by $A_{-a,b}$. For fixed $a,b>0$, $A_{-a,b}$ is a local set of the GFF such that: the value of $h_A$ inside each connected component of $D\setminus A$ is constant with value either $-a$ or $b$; and that is thin in the sense that for all $f$ smooth we have $(\Gamma_A,f) = \int_{D \backslash A} f(z)h_A(z) \, dz$. The prime example of such a set is CLE$_4$ coupled with the Gaussian free field as $A_{-2\lambda, 2\lambda}$, where $\lambda$ is an explicit constant equal to $\lambda=\pi/2$ in our case \cite{MS,ASW}. In analogy with stopping times, they correspond to exit times of Brownian motion from the interval $[-a,b]$. We recall the main properties of two-valued sets: 

\begin {propn}
\label {twdesc2}
Let us consider $-a < 0 < b$. 
\begin {enumerate}
\item
When $a+b  < 2 \lambda$, there are no local sets of $\Gamma$ with the characteristics of $A_{-a,b}$.
\item
When $a+b \ge 2 \lambda$, it is possible to construct  $A_{-a,b}$ coupled with a GFF $\Gamma$. Moreover, the sets $A_{-a,b}$ are
\begin{itemize} 
	\item Unique in the sense that if $A'$ is another local set coupled with the same $\Gamma$,  such that for all $z \in D$, $h_{A'} (z) \in \{ -a ,b \}$ almost surely and $A'$ is thin in the sense above, 
	then $A' = A_{-a, b}$ almost surely.  
	\item Measurable functions of the GFF $\Gamma$ that they are coupled with.
	\item Monotonic in the following sense: if $[a,b] \subset [a', b']$ and $-a < 0 < b$ with $b+a \ge 2\lambda$,  then almost surely, $A_{-a,b} \subset A_{-a', b'}$. 
	\item $A_{-a,b}$ has almost surely Lebesgue measure 0.
	\item For any $z$, $\log \CR(z;D\backslash A_{-a,b})-\log \CR(z;D)$ has the distribution of the hitting time of $\{-a,b\}$ by a standard Brownian motion.
\end{itemize}

\end {enumerate}
\end {propn} 

Another nice class of local sets are those that only take one value in the complement of $A$. We call them first passage sets and denote them by $A_{a}$ (if they only take the value $a$). These correspond to one-sided hitting times of the Brownian motion: hence the name. They are of interest in describing the geometry of the Gaussian free field and are treated in more detail in \cite{ALS}. Here, we only provide one working definition and refer to \cite{ALS} for a more intrinsic definition, uniqueness and other properties not needed in the current paper. 

\begin{defn}[First passage set]\label{def::fps}
Take $a\geq 0$. We say that $A_{a}$ is the first passage set (FPS) of a GFF $\Gamma$, with height $a$, if it is given by $\overline{\bigcup_n A_{-n,a}}$.
\end{defn}

We need a few properties of these sets. The first follows from the definition, the second and third from calculations in \cite{ASW} Section 6:
\begin{itemize}
\item We have that $\Gamma_{A_{a}} = a- \nu_{a}$, where $\nu_a$ is a positive measure supported on $A_{a}$;
\item $A_{a}$ has zero Lebesgue measure;
\item For any $a_n \rightarrow \infty$ we have that $\overline{\bigcup A_{a_n}} = \bar{D}.$
\end{itemize}
Note that because the circle-average of the GFF $(\Gamma, \rho_z^\eps)$ is a.s. well-defined for all $z \in D$, $\eps > 0$ simultaneously, it also means that $(\nu_a,\rho_z^\eps)$ is a.s. well-defined and positive for all $z,\eps$ as above.
	
In fact these three properties characterize $A_{a}$ uniquely \cite{ALS}. However, in this paper we only need a weaker uniqueness statement that is a consequence of the following lemma:

\begin{lemma}\label{lem:fps}
Denote $A^1 = A_{-a,a}$ with $a \geq \lambda$ and define iteratively $A^n$ by exploring copies of $A_{-a,a}$ in each connected component of the complement of $A^{n-1}$. Then, almost surely for a dense countable set $z \in D$ the following holds: for $k\in \N$, let $n_z$ be the first iteration when $h_{A^{n_z}}(z) = ak$, the connected component  $D\backslash A^{n_z}$ containing $z$ is equal to the connected component of $D\backslash A_{ak}$ containing $z$. 
\end{lemma}

\begin{proof}
The proof follows from the uniqueness of two-valued sets $A_{-a,b}$. Indeed, construct sets $B^n$ by taking $B^1 = A^1$ and then repeating the construction of $A^i$ only in the components where the value of $h_{B^n}$ is not yet $ak$. Thus, by construction $B^n \subset A^n$. Moreover, for any $z$ up to and including the first iteration where $\Gamma_{B^k}(z) = ak$, the connected component of the complement of $A^n$ and $B^n$ containing $z$ coincide.

	Now, note that for a fixed $z \in D$, $n_z$ is almost surely finite. Thus it suffices to prove that for all $n \in \N$, the set $B^n$ is contained in $A_{-\lceil an \rceil,ak}$ and that all connected components of $D\backslash B^n$ where $h_{B^n}$ takes the value $ak$ are connected components of $D\backslash A_{-\lceil an \rceil,ak}$ where $h_{A_{-\lceil an \rceil,ak}}$ is equal to $ak$. To see this, first note that $h_{B^n} \in \{-an, -a(n-1), \dots, ak\}$. In particular, in each connected component where $h_{B^n} = c \notin \{-\lceil an \rceil,ak\}$ we can construct the two-valued sets $A_{-\lceil an \rceil-c,ak-c}$. This gives us a local set $\tilde B$ s.t. $h_{\tilde B}$ takes only values in $\{-\lceil an \rceil, k\}$. It is also possible to see that $\tilde B$ is thin, by noting that inside each compact set its Minkowski dimension is smaller than $2$ (e.g. see \cite[Proposition 4.3]{Se}). Then, by uniqueness of the two-valued sets, Lemma \ref{twdesc2}, $\tilde B$ is equal to $A_{\lceil an \rceil,k}$. To finish, notice that the connected components of $D\backslash B^n$ where $h_{B^n}$ took the value $ak$ are also connected components of $\tilde B$ with the same value.
\end{proof}

In particular, from this lemma it follows that we can also construct $A_{a}$ in a different way: denote $A^1 = A_{-a,a}$ and define $A^2$ by iterating independent copies of $A_{-a,a}$ in each component of the complement of $D\setminus A^1$ where $h_{A_1} \neq a$. Repeat this procedure again in all components of the complement for which the value still differs from $a$. This iteration gives an increasing sequence of local sets $A^n$, whose limit is equal to $A_{a}$. For a concrete example, one could take $A_{-2\lambda, 2\lambda}$ to be equal to CLE$_4$ in its coupling with the GFF, and the above procedure would yield $A_{2\lambda}$. In fact the sets $(A_{2\lambda n})_{n\in \N}$  are exactly the sets that the author \cite{Ai} proposes as a basis for the construction of the Liouville measure.

\section{Overview of the Liouville measure and loop constructions of \cite{Ai}}

There are many ways to define the Liouville measure in the subcritical case, the differences amounting to how one approximates the underlying GFF. We will first describe the approximations using circle averages in the subcritical case. Then we will discuss the critical regime, and finally present the nested-loop based constructions from \cite{Ai} that are conjectured to give the Liouville measure. From now on we will set $D=\D$ for simplicity.

\subsection{Subcritical regime}

Let us recall that we denote $\Gamma_\eps(z) = (\Gamma,\rho^\eps_z)$ the $\eps$-circle average of the GFF around the point $z$ as before. It is known that $(\Gamma_\eps(z):\eps\geq 0, z\in \D)$ is a continuous Gaussian processes that converge to $\Gamma$ a.s. in the space of distributions as $\eps \rightarrow 0$. Thus, one can define approximate Liouville measures on $\mathcal{B}(\D)$ by \[\mu^\gamma_\eps(dz) := \eps^{\frac{\gamma^2}{2}} \exp\left(\gamma \Gamma_{\eps}(z)\right)dz.\]
In the subcritical regime we have the following result \cite{DS, Ber}:

\begin{thm}\label{thm::subcritical_mollified}
For $\gamma < 2$ the measures $\mu^\gamma_\eps$ converge to a non-trivial measure $\mu^\gamma$ weakly in probability. Moreover, for any fixed Borel set $\Os \subseteq \D$ we have that $\mu^\gamma_\eps(\Os)$ converges in $L^1$ to $\mu^\gamma(O)$.
\end{thm}
In fact it is known that the measure is also unique, in the sense that the same limit can be obtained using any sufficiently nice mollifier instead of the circle average. We will show that the approximations using local sets give the same measure.

\subsection{Critical regime}\label{sec:cr}

It is known that for $\gamma \geq 2$, the measures $\mu^\gamma_\eps$ converge to zero \cite{RV}. To define the critical measures an additional renormalization is therefore required. One way to do it is to use the so-called derivative martingale, originating from studies on branching random walks. Define
\begin{equation}\label{eqn::nu_eps}
 \nu_\eps(dz):= \left. \frac{\partial }{\partial \gamma} \right | _{\gamma=2}\mu_\eps^\gamma (dz) =(-\Gamma_\eps(z)+2 \log(1/\eps))\eps^2\exp\left(2\Gamma_\eps(z)\right) dz\end{equation}
It has been recently shown in \cite[Theorem 1.1]{EP} that $\nu_{\eps}$ converges weakly in probability to a non-trivial limiting measure $\mu_2'$ as $\eps\to 0$. Moreover, $\mu_2'$ coincides with the critical Liouville measure defined in \cite{DSRV,DSRV2}. We will again show that the approximations using local sets converge towards same measure.

Another way to define the critical measure is to use the so-called Seneta-Heyde renormalization \cite{AiSh, DSRV2}. In the case of the circle-average process the approximating measures would be defined as:
\[\bar{\nu_\eps}(dz) := \sqrt{\log 1/\eps}\mu^2_\eps(dz).\]
It has been shown \cite{HRVdisk, JS} that $\bar{\nu_\eps}$ converges in probability to $\sqrt{\frac{2}{\pi}}\mu'_2$ as $\eps \rightarrow 0$. We will prove an analogous result in our setting.

\subsection{Measures constructed using nested loops}
\label{sectLM}
In \cite{Ai} the author proposes a construction of measures, analogous to the Liouville measure, using nested conformally-invariant loop ensembles. We will now describe it in a concrete context that is related to this paper. 

Consider a CLE$_4$, and inside each loop toss an independent fair coin. Keep the loops with heads on top, and sample new CLE$_4$ loops in the others. Also toss new independent coins inside these loops. Keep track of all the coin tosses for each loop and repeat the procedure inside each loop where the number of heads is not yet larger than the number of tails. Define the resulting set as $\tilde A^1$. Now define $\tilde A^k$ iteratively by sampling an independent copy of $\tilde A^1$ inside each connected component of $\D \setminus \tilde A^{k-1}$.

For any Borelian $\Os\subseteq \D$ we can now define 

\begin{equation}\label{aic}
\tilde M^\gamma_k(\Os) = \frac{1}{\mathbb{E}\big[\CR(0, \D \setminus \tilde A^1)^{\gamma^2/2}\big]^k}\int_{\Os \cap \D \backslash \tilde A^k}\CR(z, \D \setminus \tilde A^k)^{\frac{\gamma^2}{2}}dz
\end{equation}

It is shown in \cite{Ai} that for $\gamma < 2$ the measures defined by $\tilde M^\gamma_k$ converge weakly almost surely to a non-trivial measure $\tilde M^\gamma$. It is also conjectured there that the limiting measures coincide with the Liouville measures $\mu^\gamma$. We will prove this statement below.

It is further proved in \cite{Ai} that for $\gamma \geq 2$, these measures converge almost surely to zero. In the critical case, however, one can again define a derivative martingale $\tilde D^\gamma_k$ by taking a derivative with respect to $-\gamma$. In other words one sets:
\[ \tilde D^\gamma_k(\Os) = - 2\frac{\partial}{\partial \gamma} \tilde{M}^\gamma_k(\Os)\] 
(we include the factor 2 here to be consistent with the definition in \cite{Ai}). It is shown in \cite{Ai} that the measures $\tilde D_k:=\tilde D^2_k$ converge to a non-trivial positive measure $\tilde D_\infty$. In this paper, we prove that $\tilde D_\infty=2\mu'_2$.

\section{Local set approximations of the subcritical Liouville measure} \label{section::subcritical_localsets}

In this section we prove that one can approximate the Liouville measure of a GFF in a simply connected domain using increasing sequences of local sets $(A^n)_{n\in \N}$  with $\overline{\bigcup A^n} = \bar{\D}$. In particular, the measure constructed in \cite{Ai} will fit in our framework and thus it agrees with the Liouville measure. In fact, for simplicity, we first present the proof of convergence in this specific case.

First, recall that we denote by $h_A$ the harmonic function given by the restriction of $\Gamma_A$ to $\D\setminus A$. For any local set $A$ with Lebesgue measure 0 and bounded $h_A$, we define for any Borelian set $\Os\subseteq \D$:
\[M^\gamma(\Os,A):=\int_{\Os} e^{ \gamma h_A} \CR(z; \D\setminus A)^{\gamma^2/2} \, dz. \]
Notice that as $h_A$ is bounded, we can define it arbitrarily on the 0 Lebesgue measure set $A$.

\begin{propn}\label{propn::Ea} Fix $\gamma\in [0,2)$. For $a > 0$, let  $A_{a}$   be the $a$-FPS of $\Gamma$ and $\mu^\gamma$ be the Liouville measure defined by $\Gamma$. Then for each Borelian set $\Os\subseteq \D$ (including $\Os=\D$), \[M_a^\gamma(\Os) := M^\gamma(\Os, A_{a})=e^{\gamma a}\int_{\Os}\CR(z;\D\backslash  A_a)^{\gamma^2/2}dz\] is a martingale with respect to $\F_{A_{a}}$ and converges a.s. to $\mu^\gamma(\Os)$ as $a \to \infty$. Thus, a.s. the measures $M_a^\gamma$ converge weakly to $ \mu^\gamma$.
\end{propn}

Before the proof, let us see how it implies that the martingales defined in \cite{Ai} converge to the Liouville measure:

\begin{cor}\label{IdPhi}
The martingales $\tilde M^\gamma_k$ defined in \cite{Ai} converge weakly a.s. to  $ \mu^\gamma$. 
\end{cor}
\begin{proof}
As a consequence of Lemma \ref{lem:fps}, the fact that $A_{-2\lambda,2 \lambda}$ has the law of CLE$_4$ and the fact that the value of its corresponding harmonic function is independent in each connected component of $\D \backslash A_{-2\lambda,2 \lambda}$ \cite{MS,ASW}, we see that $\tilde A^1$ of Section \ref{sectLM} is equal in law to $A_{2\lambda}$. Furthermore, the sequence $(\tilde A^k)_{k \in \N}$ has the same law as the sequence $(A_{2\lambda k})_{k \in \N}$.

 Now, by the iterative construction and conformal invariance the random variables 
\[\log \CR(0, \D \setminus \tilde A^i) - \log \CR(0, \D \setminus \tilde A^{i-1}) \]
with $A^0 = \emptyset$ are i.i.d. Thus, $\mathbb{E}\big[\CR(0, \D \setminus \tilde A^1)^\frac{\gamma^2}{2}\big]^{k} = \mathbb{E}\big[\CR(0, \D \setminus \tilde A^k)^\frac{\gamma^2}{2}\big]$.

Moreover, it is known from \cite{SSW, ASW} that $-\log \CR(0, \D \setminus \tilde A^k)$ corresponds precisely to the hitting time of $k\pi$ by a standard Brownian motion started from zero. In our case, when $2\lambda = \pi$, we therefore see that \[e^{\gamma 2\lambda k} = \mathbb{E}\big[\CR(0, \D \setminus \tilde A^1)^\frac{\gamma^2}{2}\big]^{-k}.\] 
Furthermore, since $Leb(A_{2\lambda})=0$ implies that $M_a^\gamma(\Os \cap A_{2\lambda}) = 0$, we have that $M_{2\lambda k}^\gamma$ agrees with the measure $\tilde M^\gamma_k$ defined in \eqref{aic}. Hence Proposition \ref{propn::Ea} confirms that the limit of $\tilde{M}_k^\gamma$ corresponds to the $\gamma$- Liouville measure.

\end{proof}

\begin{remark0}\label{IdWC}
In order to avoid repetition, we recall here as a remark the standard argument showing that the almost sure weak convergence of measures is implied by the almost sure convergence of $M_a^\gamma(\Os)$ over all boxes $\Os\subset \D$ with dyadic coordinates, and the convergence of the total mass $M_a^\gamma(\D)$. This follows from two observations. Firstly, since all the approximate measures $M_a^\gamma$ have zero mass on the boundary $\partial \D$, we can extend them to Radon measures on $\bar{\D}$. We do this because any closed sub-space of Radon measures on $\D$ with uniformly bounded mass is compact (with respect to the weak topology), and therefore we have subsequential limits in this space. Secondly, the boxes $\Os\subset \D$ with dyadic coordinates generate the Borel $\sigma$-algebra on $\D$. This identifies any subsequential limit as a measure on $\bar{\D}$ uniquely (and it must have zero mass on the boundary) since we know that the total masses $M_a^\gamma(\D)=M_a^\gamma(\bar{\D})$ converge.

Notice that $\mu^\gamma_\epsilon$ does not converge in the strong topology of measures. This follows from the fact that almost surely $\mu^\gamma$ is not absolutely continuous with respect to Lebesgue measure.
\end{remark0}

\begin{proofof}{Proposition \ref{propn::Ea}} 
By Remark \ref{IdWC}, it suffices to prove the convergence statement for $M_a^\gamma(\Os)$ (with $\Os\subseteq \D$ arbitrary).	When $\gamma\in [0,2)$, we know that $\mu^\gamma_\eps(\Os)\to \mu^\gamma(\Os)$, in $\mathcal L^1$ as $\eps \to 0$, where $\mu^\gamma_{\eps}$ is as in Theorem \ref{thm::subcritical_mollified}. Thus, \[\E{\mu^\gamma(\Os)\mid \F_{A_{a} }} = \lim_{\eps \to 0} \E{\mu_\eps^\gamma(\Os)\mid \F_{A_{a} }}.\] 	
	The key is to argue that 
	\begin{equation}\label{eq:cond}
	\lim_{\eps \to 0}\E{ \mu_\eps^\gamma(\Os)\mid \F_{A_{a} }} = M_a^\gamma(\Os).
	\end{equation}
	 Then $M_a^\gamma(\Os) = \E{\mu^\gamma(\Os)\mid \F_{A_{a} }}$ and we can conclude using the martingale convergence theorem and the fact that $\overline{\bigcup A_{a}} = \bar{\D}$ (so that $\mathcal{F}_\infty=\sigma(\Gamma)$). 
	
	To prove \eqref{eq:cond}, define $A_{a}^\eps$ as the $\eps$-enlargement of $A_{a}$. By writing $\Gamma=\Gamma_{A_{a}}+\Gamma^{A_{a}}$ and using that $(\Gamma_{A_{a}},\rho_\eps^z)= a$ for any $z \in \D \backslash A_{a}^\eps$, we have
	\[\E{\left.\int_{\Os \backslash A_{a}^\eps} \e^{\gamma(\Gamma, \rho_\eps^z)}\eps^{\gamma^2/2} \, dz\right| \F_{A_{a}}}= \int_{\Os \backslash A_{a}^\eps} \e^{\gamma a}\eps^{\gamma^2/2} \E{\left.\e^{(\Gamma^{A_{a}}, \rho_\eps^z)}\right| \F_{A_{a}}}dz\]
	Using \eqref{eqn::condmeanGammaA} we recognize that the right hand side is just $M_a^\gamma(\Os \backslash A_{a}^\eps)$.

On the other hand, $(\Gamma_{A_{a}},\rho_\eps^z) \leq a$ for any $z$, and the conditional variance of $(\Gamma^{A_a},
	\rho_\eps^z)$ given $\F_{A_{a}}$ is less than that of $(\Gamma, \rho_\eps^z)$. Thus we can bound 
\[\eps^{\gamma^2/2}\E{\left.\e^{(\Gamma^{A_{a}} + \Gamma_{A_{a}}, \rho_\eps^z)}\right|\F_{A_{a}}} \leq \e^{\gamma a}\CR(z,\D)^{\gamma^2/2}\]
	and it follows (using that $\CR(z, \D) \leq 1$ for all $z\in \D$) that
	\[\E{\left.\int_{\Os\cap A_{a}^\eps} \e^{\gamma(\Gamma, \rho_\eps^z)}\eps^{\gamma^2/2} \, dz\right| \F_{A_{a}}} \leq Leb(A_{a}^\eps) e^{\gamma a}.\]
	Since $A_{a}$ has zero Lebesgue measure, we have $Leb(A_{a}^\eps) = o_\eps(1)$. This concludes \eqref{eq:cond} and the proof.
\end{proofof}

We now state a more general version of this result, which says that one can construct the Liouville measure using a variety of local set approximations. The proof is a simple adaptation of the proof above. We say that a generalized function $T$ on $\D$, for which the circle-average process $T_\eps(z) := (T,\rho_z^\eps)$ can be defined, is bounded from above by $K$ if for all $z \in D$ and $\eps >0$, we have that $T_\eps(z) \leq K$. 

\begin{propn}\label{propn::mainresult_withF}
	Fix $\gamma\in [0,2)$ and let $(A^n)_{n \in \N}$ be an increasing sequence of local sets for a GFF $\Gamma$ with $\overline{\bigcup_{n \in \N} A^n} = \overline \D$. Suppose that almost surely for all $n \in \N$, we have that $Leb(A^n)=0$ and that $\Gamma_{A^n}$ is bounded from above by $K_n$ for some sequence of finite $K_n$. Then for any Borel $\Os \subseteq \D$ (including $\Os=\D$), $M_n^\gamma(\Os)$ defined by 
	\[ M_n^\gamma(\Os)= \int_{\Os} e^{\gamma h_{A^n}(z)}\CR(z; \D\setminus A^n)^{\gamma^2/2} \, dz \]
	is a martingale with respect to $\{\mathcal{F}_{A^n}\}_{n>0}$ and 
	\[ \lim_{n\to \infty} M_n^\gamma(\Os) = \mu^\gamma(\Os)  \, \text{a.s.}\]
	where $\mu^\gamma$ is the Liouville measure defined by $\Gamma$. Thus, almost surely the measures $M_n^\gamma$ converge weakly to $ \mu^\gamma$.
	\end{propn}
Let us mention two natural sequences of local sets for which this proposition applies. The first is when we take $a_n,b_n\nearrow \infty$ and study the sequence $(A_{-a_n,b_n})_{n\in \N}$. The second is when we take the sequence $(A_{-a,b}^n)_{n\in \N}$ for some $a,b>0$, where $A^n_{-a,b}$ is defined by iteration \footnote{We set $A^1_{-a,b}=A_{-a,b}$ and define $A_{-a,b}^n$ by sampling the $A_{-a,b}$ of $\Gamma^{A^{n-1}_{-a,b}}$ inside each connected component of $D\backslash A_{-a,b}^{n-1}$}. Note that in the case where $a=b=2\lambda$, we recover the result described in the introduction for the iterated CLE$_4$.

Observe that whereas our martingale agrees with the one given in \cite{Ai} for the case of first-passage sets, for any cases where $h_{A^n}$ can take more than one value, the martingales are in fact different. Yet, we can still identify the limit of the martingale $\tilde M^\gamma_n(\Os)$ of \cite{Ai}, corresponding to an iterated CLE$_4$ (i.e. $(\text{CLE}_4^n)_{n\in \N}$.) In this case Aidekon's martingale converges in distribution to $\eta^\gamma(\Os) := \E{\mu^\gamma(\Os) | \F_\infty}$, where $\mu^\gamma$ is the Liouville measure and $\F_\infty$ is the $\sigma$-algebra containing only the geometric information from all iterations of the CLE$_4$. This $\sigma$-algebra is strictly smaller than  $\mathcal{F}_{A_{-2\lambda,2\lambda}^n}$, which also contains information on the labels of CLE$_4$ in its coupling with the GFF. It is not hard to see that $\eta^\gamma$ is not equal to $\mu^\gamma$. 

\section{Generalizations}
In this section, we describe some other situations where an equivalent of Proposition \ref{propn::mainresult_withF} can be proven using the same techniques as the proof of Proposition \ref{propn::Ea}. In the following we do not present any new methods, but focus instead on announcing the propositions in context, so that they may be used in other works. We also make explicit the places where the results are already, or may in the future, be used.
\subsection{Non-simply connected domains and general boundary conditions.}
Here we consider the case when $\Gamma$ is a GFF in an n-connected domain $D\subseteq \D$ (for more context see \cite{ALS}). First, let us note that in this set-up \eqref{eqn::condmeanGammaA} becomes
\[
\E{\ \eps^{\frac{\gamma^2}{2}} \exp\left(\gamma (\Gamma,\rho^\eps_z)\right)}\left\{ \begin{array}{l l}
=  e^{-\frac{\gamma^2}{2} \tilde G_{D}(z,z)} & \text{ if } d(z,\partial D)\geq \eps,\\
\leq 1 & \text{ if } d(z,\partial D)< \eps,
\end{array}\right.\]
where we write $G_D(z,w) = -\log |z-w| + \tilde G_D(z,w)$, i.e. for any $z \in D$, $\tilde G_{D}(z,\cdot)$, is the bounded harmonic function that has boundary conditions $\log(|z-w|)$ for $w\in \partial D$. Additionally, if we work with local sets $A$ such that all connected components of $A\cup \partial D$ contain an element of $\partial D$, then Lemma \ref{BPLS} will hold. All local sets we refer to here are assumed to satisfy this condition. These facts and assumptions are  enough to prove the following proposition:
\begin{propn}\label{propn::mainresultgeneral}
	Fix $\gamma\in [0,2)$ and let $(A^n)_{n \in \N}$ be an increasing sequence of local sets for a GFF $\Gamma$ with $\overline{\bigcup_{n \in \N} A^n} = \bar{D}$. Suppose that almost surely for all $n \in \N$, we have that $Leb(A^n)=0$ and that $\Gamma_{A^n}$ is bounded from above by $K_n$ for some sequence of finite $K_n$. Then for any Borel $\Os \subset D$, $M_n^\gamma(\Os)$ defined by 
	\[M_n^\gamma:=\int_{\Os} e^{ \gamma h_{A_n}(z)-\frac{\gamma^2}{2} \tilde G_{D\backslash A_n}(z,z)} \, dz,  \]
	is a martingale with respect to $\{\mathcal{F}_{A^n}\}_{n>0}$ and 
	\[ \lim_{n\to \infty} M_n^\gamma(\Os) = \mu^\gamma(\Os)  \, \text{a.s.}\]
	where $\mu^\gamma$ is the Liouville measure defined by $\Gamma$. Thus, almost surely the measures $M_n^\gamma$ converge weakly to $ \mu^\gamma$. \end{propn}
The equivalent of the sets $A_{-a,b}$ and $A_{a}$ are defined in $n$-connected domains in \cite{ALS} and it is easy to see that their iterated versions satisfy the hypothesis of Proposition \ref{propn::mainresultgeneral}. In particular, the above construction allows the authors in \cite{ALS} to prove that the measure $\Gamma_{A_{a}}$ is a measurable function of $A_{a}$.

\subsection{Dirichlet-Neumann GFF}
In this section we take $\Gamma$ to be a GFF with Dirichlet-Neumann boundary conditions in $\D^+=\D\cap \HH$. That is, $\Gamma$ satisfies \eqref{GFF}, with $G_D$ replaced by $G_{\D^+}$: the Green's function in $\D^+$ with Dirichlet boundary conditions on $\partial \D  $ and Neumann boundary conditions on $[-1,1]$. To be more specific, we set $G_{\D^+}(x,y)=G_{\D}(x,y)+G_{\D}(x,\bar y)$, with $G_{\D}$ as in Section \ref{sectGFF}. Then $G_{\D^+}(x,y)\sim \log(1/|x-y|)$ as $x\to y$ in the interior of $\D^+$ and $G_{\D^+}(x,y)\sim 2\log(1/|x-y|)$ when $y \in (0,1)$.

Let $A$ be a closed subset of $\overline{\D^+}$. Suppose that $\Gamma$ is a Dirichlet-Neumann GFF in $\D^+\backslash A$ with Neumann boundary conditions on $[-1,1] \backslash A$ and Dirichlet boundary conditions on the rest of the boundary. Let $z\in [-1,1]$ and  define $\varrho_z^\eps$ to be the uniform measure on $\partial B(z,\eps)\cap \D^+$. Then, in this set-up \eqref{eqn::condmeanGammaA} becomes 
\begin{equation}\label{eqn::condmeanGammaAB1B2}
\E{ \eps^{\gamma^2/4} \exp\left(\frac{\gamma}{2} (\Gamma,\varrho^\eps_x)\right)}\left\{ \begin{array}{l l}
= \CR(x;\D\backslash \breve{A})^{\gamma^2/4} & \text{ if } d(z,\partial (\D\backslash \breve{A}))\geq \eps,\\
\leq 1& \text{ if } d(z,\partial (\D\backslash \breve{A}))< \eps.
\end{array}\right.
\end{equation}
Here we set $\breve A:= A\cup \bar{A}$ for $\bar A=\{z \in \C: \bar z \in A\}$.

There is also a notion of local sets for this Dirichlet-Neumann GFF. We say that $(\Gamma,A,\Gamma_A) $ describes a local set coupling if, conditionally on $(A,\Gamma_A)$, $\Gamma^A:=\Gamma-\Gamma_A$ is a GFF with Neumman boundary conditions on $[-1,1]\backslash A$ and Dirichlet on the rest. For connected local sets such $\partial \D^+\cup A$ is connected, Lemma \ref{BPLS} still holds (by the same proof given for the 0-boundary GFF).

We are interested in the boundary Liouville measure on $[-1,1]$. Take $\gamma<2$, $\eps >0$ and a Borel set $\Os \subseteq [-1,1]$. We define the approximate boundary Liouville measures as follows:
\[
\upsilon^\gamma_\eps(\Os):=\eps^{\gamma^2/4}\int_{\Os}\exp\left (\frac{\gamma}{2}(\Gamma,\varrho_x^\eps) \right )dx,\]
 	where here $dx$ is the Lebesgue density on $[-1,1]$.
It is known (see \cite{DS, Ber}) that $\upsilon^\gamma_\eps \to \upsilon^\gamma$ in $\mathcal L^1$ as $\eps \to 0$. Moreover, it is also easy to see that $\upsilon^\gamma$ is a measurable function of $\F_{[-1,1]}$ - this just comes from the fact that the Dirichlet GFF contains no information on the boundary. Thus, we have all the necessary conditions to deduce the following Proposition using exactly the same proof as in Section \ref{section::subcritical_localsets}.

\begin{propn}\label{propn::mainresultDN}
	Fix $\gamma\in [0,2)$ and let $(A^n)_{n \in \N}$ be an increasing sequence of local sets for a GFF $\Gamma$ with $\overline{\bigcup_{n \in \N} A^n} \supseteq [-1,1]$. Suppose that almost surely for all $n \in \N$, we have that $Leb_{[-1,1]}(A_n)=0$ and that $\Gamma_{A^n}$ restricted to $A^n$ is bounded from above by $K_n$ for some sequence of finite $K_n$. Then for any Borel $\Os \subseteq [-1,1]$, $M_n^\gamma(\Os)$ defined by 
	\[M^\gamma_n(\Os):=\int_{\Os} e^{ \frac{\gamma}{2} h_{A_n}(z)}\CR(z;\D\backslash \breve A_n)^{\frac{\gamma^2}{4}} \, dz \]
	is a martingale with respect to $\{\mathcal{F}_{A^n}\}_{n>0}$ and 
	\[ \lim_{n\to \infty} M_n^\gamma(\Os) = \upsilon^\gamma(\Os)  \, \text{a.s.}\]
	where where $\mu^\gamma$ is the boundary Liouville measure defined by $\Gamma$. Thus, almost surely the measures $M_a^\gamma$ converge weakly to $ \upsilon^\gamma$.
	 \end{propn}

 It has recently been proven in \cite{QW} that sets satisfying the above hypothesis do exist, and that they can be used to couple the Dirichlet GFF with the Neumman GFF. Let us describe some concrete examples of these sets. If $\Gamma$ is a Dirichlet-Neumman GFF, then in \cite{QW} it is shown that there exists a (measurable) thin local set $\tilde A(\Gamma)$ of the GFF such that: 
 \begin{itemize}
 \item $\tilde A(\Gamma)$ has the law of the trace of an SLE$_4(0;-1)$ going from $-1$ to $1$
 \item $h_{\tilde A(\Gamma)}$ is equal to $0$ in the only connected component of $\D ^+\backslash \tilde A(\Gamma)$ whose boundary intersects $\partial \D \cap \HH$ 
 \item in the other connected components, $h_{\tilde A(\Gamma)}$ is equal to $\pm 2\lambda$, where conditionally on $\tilde A(\Gamma)$ the sign is chosen independently in each component.
\end{itemize}

There are two interesting sequences of local sets we can construct using this basic building-block. The first one is the boundary equivalent of $(A_{-2\lambda,2\lambda}^n)_{n\in \N}$, and the second is the boundary equivalent of $(A_{2\lambda n})_{n\in \N}$. 
The first one is also described in \cite[Section 3]{QW}. The construction goes as follows: choose $A^1=\tilde A(\Gamma)$ and construct $A^n$ by induction. In the connected components $O$ of $\D\backslash A^n$ that contain an interval of $\R$, we have that $\Gamma^{A_n}$ restricted to $O$ is a Dirichlet-Neumman GFF (with Neumann boundary condition on $\R\cap \partial O$). Thus, by conformal invariance we can explore the set $\tilde A(\Gamma\mid_O)$ in each such component $O$. We define $A^{n+1}$ to be the closed union of $A^{n}$ with $\tilde A(\Gamma\mid_O)$ over all explored components $O$. Note that $h_{A^{n}}\in \{2\lambda k\}$ where $k$ ranges between $-n$ and $n$. It is also not hard to see that $A^n$ is thin (it follows from the fact that $h_A\in \mathcal{L}^1(\D\backslash A)$ and that for any compact set $K\subseteq \D^+$ the Minkowski dimension of $A^n\cap K$ is a.s. equal to $3/2$, see e.g. \cite[Proposition 4.3]{Se}). Thus we deduce that $\Gamma_{A^n}\leq 2\lambda n$. Additionally, note that by adjusting \cite[Lemma 6.4]{MS}, we 
obtain from the construction of $A^1$ that for any $z\in (-1,1)$ the law of $2(\log(\CR^{-1}(z,\D\backslash \breve A^1))-\log(\CR^{-1}(z,\D))$ is equal to the first time that a BM exits $[-2\lambda,2\lambda]$. It follows that for all $n\in \N$,  $Leb_{\R}(A_n \cap [-1,1]) = 0$ and also $\overline{\bigcup_{n \in \N} A^n} \supseteq [-1,1]$. Hence we see that the sequence $(A^n)_{n\in \N}$ satisfies the conditions of Proposition \ref{propn::mainresultDN}.

For the second sequence of local sets, take $B^1=\tilde A(\Gamma)$ and define $B^{n+1}$ to be the closed union of $B^n$ with all $\tilde A(\Gamma\mid_O)$ such that $O$ is a connected component of $\D\backslash B^n$, $h_{B^n}\mid _O\leq 2 \lambda$ and $\partial O$ contains an interval of $\R$. Denote $A^1(\Gamma)$ the closed union of all the $B^n$. Due to the fact that $B^n$ are BTLS with $h_{B^n}\leq 2\lambda$ on $[-1,1]$, we have that $\Gamma_{A^1}$ restricted to $[-1,1]$ is smaller than or equal to $ 2\lambda$. Additionally, note that $2(\log(\CR^{-1}(z,\D\backslash \breve A^1))-\log(\CR^{-1}(z,\D))$ is distributed as the first time a BM hits $2\lambda$. Now, we iterate to define $A^n(\Gamma)$ as the closed union of $A^{n-1}(\Gamma)$ and $A^1(\Gamma\mid _O)$, where $O$ ranges over all connected components of $\D^+\backslash A^{(n-1)}$ containing an interval of $\R$. The sequence $(A^n)_{n\in \N}$ satisfies the condition of Proposition \ref{propn::mainresultDN}. Note that in this case the martingale simplifies and contains only information on the geometry of the sets $A^n$:
\[
M^\gamma_n:= e^{\gamma 2\lambda n}\int_{\Os} CR(z;\D\backslash \breve A^n)^{\gamma^2/4}dz.
\]
The fact that this martingale is a measurable function of $A^n$ allows us to use the same techniques as in \cite{ALS} to prove that the measure $2\lambda n-\Gamma_{A^n}$ on $\R$ is a measurable function of $A^n$.

It is also explained in \cite{QW} that the sets $A^n$ we have just constructed, and the definition of the boundary Liouville measure using them, might help to reinterpret an SLE-type of conformal welding first studied in \cite{ShZ}.

\section{Critical and supercritical regimes}

In this section it is technically simpler to restrict ourselves to the simply connected case and to study a special family of sequences of local sets. Namely, we assume that our sets $A^n$ are formed by iterating a first passage set $A_a$ for some $a>0$, in other words $A^n = A_{an}$. With some extra work, the results can be seen to hold in a more general setting.

We first show that the martingales defined in Section \ref{section::subcritical_localsets} converge to zero for $\gamma \geq 2$. Then, in the critical case $\gamma = 2$, we define a derivative martingale and show it converges to the same measure as the critical measure $\mu_2'$ from \cite{DSRV,DSRV2,EP}, and $1/2$ times the critical measure $\tilde{D}_\infty$ from \cite{Ai}. Finally, we show that we can also construct the critical measure using the Seneta-Heyde rescaling (analogous to the main theorem of \cite{AiSh}.) More precisely, for all Borelian $\Os \subseteq \D$, we have that $\sqrt{an}M^2(\Os, A^n)$ converges in probability to $\frac{2}{\sqrt{\pi}}\mu_2'(\Os)$ as $n\to \infty$.
\subsection{The martingale $M_n^\gamma$ vanishes in the (super)critical regime.}\label{sec:crz}
\begin{lemma}\label{lem::criticaltozero}
	Set $\gamma \geq 2$ and $A^n=A_{an}$ as above. 
Then  $M_n^\gamma\to 0$	almost surely.
\end{lemma}

{\begin{remark0}\label{rmk::iterated_0}
	In fact, our proof of Lemma \ref{lem::criticaltozero} works for any sequence $A^n$ of local sets such that $\overline{\bigcup_n A_n}=\bar{\D}$, and that are formed by iteration. That is, $A^1=A(\Gamma)$ is some measurable local set coupled with the GFF $\Gamma$, and $A^{n+1}$ is formed from $A^n$ by, in each component $O$ of $\D\setminus A^n$, exploring $A(\Gamma^{A^n})$
\end{remark0}


In \cite{Ai}, A\"{i}dekon also considers the critical and supercritical cases for his iterated loop measures. In particular, from his results one can obtain Lemma \ref{lem::criticaltozero} directly. We include a proof (that works in the more general setting of Remark \ref{rmk::iterated_0}) for completeness, and to introduce a change of measure technique that will be crucial in later arguments. The proof follows from a classical argument, stemming from the literature on branching random walks \cite{Lyons}, but is based on the local set coupling with the GFF.
\\

\begin{proof} 
	From \eqref{eqn::condmeanGammaA} and the iterative way that we have constructed $A^n$, we see that if $M_0^\gamma(\D)=\int_\D\CR(z,\D)^{\gamma^2/2} \, dz$, then $M_n^\gamma(\D)/M_0^\gamma(\D)$ is a mean one martingale. Let us define a new probability measure $\hat{\mathbb{P}}$ via the change of measure 
	\begin{equation}\label{com} \left.\frac{d\hat{\mathbb{P}}}{d\mathbb{P}}\right|_{\F_{A^n}} = \frac{M_n^\gamma(\D)}{M_0^\gamma(\D)}.\end{equation}
	It is well known, see for example \cite[Theorem 5.3.3]{Dur}, that in order to show that $M_n^\gamma(\D)\to 0$ almost surely under $\mathbb{P}$, it suffices to prove that $\limsup_n M_n^\gamma(\D)=+\infty$ a.s. under $\hat{\mathbb{P}}$. 
	
	To show this we actually consider a change of measure on an enlarged probability space. Define a measure $\mathbb{P}^*$ on $(\Gamma, (A^n)_n, Z)$ by sampling $(\Gamma, (A^n)_n)$ from $\mathbb{P}$ and then independently, sampling a random variable $Z\in \D$ with law proportional to Lebesgue measure. Note that under $\mathbb{P}^*$ the process 
	\[ \xi_n =  \e^{\gamma h_{A^n}(Z)-\gamma^2/2 \log \CR^{-1}(Z,\D\setminus A^n)}\]
	is a martingale with respect to the filtration $\F_{A^n}^*=\F_{A^n}\vee \sigma(Z)$. Thus we can define a new probability measure $\hat{\mathbb{P}}^*$ by 
	\begin{equation}
	\label{Pstar}\left. \frac{d\hat{\mathbb{P}}^*}{d\mathbb{P}^*}\right|_{\F_{A^n}^*}:=\frac{\xi_n}{\mathbb{E}^*[\xi_0]}
	\end{equation}
	Then if $\hat{\mathbb{P}}$ is the restriction of $\hat{\mathbb{P}}^*$ to $\F_{A^n}$, $\hat{\mathbb{P}}$ and $\mathbb{P}$ satisfy (\ref{com}). Therefore it suffices to prove that under $\hat{\mathbb{P}}^*$ and conditionally on $Z$, we have $\limsup_n M_n^\gamma(\D)=+\infty$ almost surely. 
	
	Now, for any simply-connected domain $D$ we have $d(z,\partial D) \leq \CR(z,D) \leq 4d(z, \partial D)$ by K\"{o}ebe's quarter theorem. Using the triangle inequality we therefore see that for any $z'$ with $d(z,z') \leq \CR(z,D)/2$, we have that $\CR(z,D) \leq 16\CR(z',D)$. Clearly, we can lower bound $M_n^\gamma(\D)$ by the $M_n^\gamma-$ mass in the disk of radius $\CR(Z,D)/2$ around $Z$. Thus by the above comments, it suffices to prove that under $\hat{\mathbb{P}}^*$ and conditionally on $Z$, almost surely 
	\begin{equation}\label{eq:rwinfty}
	\limsup_n \e^{\gamma h_{A^n}(Z)-(\gamma^2/2+2)\log \CR^{-1}(Z,\D\setminus A^n)}=+\infty.
	\end{equation}
	To do this, we claim that under the conditional law $\hat{ \mathbb P}^*(\cdot | Z)$
	\[ h_{A^n}(Z)-\gamma\log \CR^{-1}(Z,\D\setminus A^n )\]
	is a random walk with mean zero increments (starting from $-\gamma \log \CR^{-1}(Z,\D\setminus A^n$). Notice that \eqref{eq:rwinfty} then follows, as $\gamma^2 \ge \gamma^2/2 + 2$ if $\gamma\geq 2$.
	
	To prove the claim, observe that the marginal law of $Z$ under $\hat{ \mathbb P}^*$ is proportional to $\CR(z,\D)^{\gamma^2/2}$. Moreover, the conditional law on the field can be written as \[\hat{ \mathbb P}^*(d\Gamma | Z) = \e^{\gamma h_{A^n}(Z)-\frac{\gamma^2}{2}(\log \CR^{-1}(Z,\D\setminus A^n)-\log \CR^{-1}(Z,\D))} \mathbb{P}(d\Gamma).\] As for all $\gamma$ we have that \[\E{\e^{\gamma h_{A^n}(Z)-\frac{\gamma^2}{2}(\log \CR^{-1}(Z,\D\setminus A^n)-\log \CR^{-1}(Z,\D))}}=1,\]
	by dominated convergence we can differentiate with respect to $\gamma$ to obtain that
	\[\E{(h_{A^n}(Z)-\gamma\log \CR^{-1}(Z,\D\setminus A^n))\e^{\gamma h_{A^n}(Z)-\frac{\gamma^2}{2}(\log \CR^{-1}(Z,\D\setminus A^n)-\log \CR^{-1}(Z,\D))}}=0.\]
	This says precisely that $h_{A^n}(Z)-\gamma\log \CR^{-1}(Z,\D\setminus A^n)$ is a zero mean random walk under $\hat{ \mathbb P}^*(\cdot | Z)$.
	
\end{proof}

\begin{remark0}\label{rmk::variance}
	Using the same technique but instead differentiating twice with respect to $\gamma$, we can also calculate the variance of the increments of $h_{A^n}(Z)-\gamma\log \CR^{-1}(Z,\D\setminus A^n)$ under $\hat{ \mathbb P}^*(\cdot | Z)$. In the case $\gamma=2$ the variance is equal to $1/2$.
\end{remark0}

\subsection{The derivative martingale in the critical regime}\label{sec::derivativemgale}

We now show the convergence of the derivative martingale (when $\gamma=2$, defined below) that is built from the sets $A^n=A_{an}$ for $a>0$. For any Borel set $\Os\subseteq \D$ and local set $A$, we define
\[ D^\gamma(\Os,A):= \int_{\Os} \left(-h_{A}(z) + \gamma \log \CR^{-1}(z,\D\setminus A)\right)e^{\gamma h_{A}(z)}\CR(z; \D\setminus A)^{\gamma^2/2}\, dz.\]
The rest of this section is devoted to proving the following proposition.

\begin{propn}\label{propn::derivative}
Assume that $A^n$ is given by $A_{an}$ for some $a > 0$. Then for any Borel $\Os \subset \D$ (including $\Os=\D$) we have that $ D^2(\Os, A^n)$ is a martingale and converges almost surely to a finite, positive limit $ D_\infty(\Os)$ as $n\to \infty$. In particular the signed measures $ D^2(\Os, A^n)$ converge weakly to a limiting measure that is independent of the choice of $a>0$ and agrees with the critical measure $\mu_2'$ defined in \cite{DSRV,DSRV2}, and $1/2$ times the critical measure $\tilde D_\infty$ defined in Theorem 1.3 of \cite{Ai}. In particular we confirm that $\tilde D_\infty = 2\mu_2'$. 
\end{propn}

\begin{remark0}
	\label{rmk::critical_CLE}
	With some extra work, one can also obtain the above result when $A^n$ is formed by iterating $A_{-a,a}$ for any $a>0$. Intuitively, this just follows from the explicit relationship between these sets and the corresponding first passage sets $(A_{an})_n$, described in Lemma \ref{lem:fps}. 
\end{remark0}

The fact that the above martingales converge (when $A^n=A_{an}$, which we stick to from now on) and that their limit agrees with $(1/2) \tilde{D}_\infty$ - follows directly from \cite{Ai}. Indeed, in the case $a=2\lambda$, observe that for any $\Os\subseteq \D$, twice the derivative martingale $2D^2(\Os,A_{2\lambda n})$ is equal to $\tilde D_n(\Os)$ defined in  (1.3) of \cite{Ai} (see proof of Corollary \ref{IdPhi}). Thus, we know from Theorem 1.3 of \cite{Ai} that when  we iterate $A_{2\lambda}$, the associated sequence of measures $2D^2(\cdot,A_{2\lambda n})$ converges weakly to $\tilde{D}_\infty$. Moreover, it also follows from Theorem 1.3 of \cite{Ai} that for all dyadic $s>0 $,  $ 2D^2(\cdot, A_{2\lambda sn})$ converges to the same limit. Doob's maximal inequality then implies that there exists a modification of $2D^2(\cdot,A_t)$ that also converges to $\tilde D_\infty$ as $t\to \infty$. This clearly implies the convergence of our derivative measures to $(1/2)\tilde{D}_\infty$ for any $a>0$.

Now we would like to connect the measures $\tilde{D}_\infty$ and $\mu'$. Concentrating on the case $a=1$, we set \[D_n:=2D^2(0,A_n),\] where we have included the factor 2 for consistency with \cite{Ai}.  The immediate difficulty is that the martingales $D_n(\Os)$ are not uniformly integrable (U.I.). To solve this issue we work with a certain mollified and localized approximation of $\mu_2'(\Os)$. Recall from Section \ref{sec:cr} that the mollified measures $\nu_{\eps}$ defined in \eqref{eqn::nu_eps} converge weakly in probability as $\eps\to 0$ to $\mu_2'$. To ensure uniform integrability, we work with a localized version that we call $\nu_\eps^\beta$. This family is U.I. for any $\beta$ (as shown in \cite[Proposition 3.6]{EP}) and, moreover, there almost surely exists a $\beta_0$ such that $\nu_{\eps}^\beta=\nu_{\eps}$ for all $\beta\ge \beta_0$. We then roughly follow the strategy of the proof of Proposition \ref{propn::Ea}, and show that the conditional expectation  $\nu_\eps^\beta$ w.r.t. $\F_{A_n}$ is approximately equal to $D_n(\Os)/2$. 

\medbreak

\begin{proofof}{Proposition \ref{propn::derivative} }
 Consider the circle-average approximate measures $\nu_{\eps}$ from \eqref{eqn::nu_eps}, and choose a sequence $\eps_k\to 0$ such that $\nu_\eps \to \mu_2'$ almost surely. From now on whenever we write $\eps\to 0$, it means that we are converging to $0$ via $(\eps_k)_{k\in \N}$. We set, for fixed $\Os\subseteq \D$, 
\[\nu_\eps^{\beta}(\Os)=\int_{\Os} (-\Gamma_\eps(z)+2\log(1/\eps)+\beta)\I{T_\beta(z)\leq \eps}
\I{\eps\le d(z,\partial D)}\e^{2\Gamma_\eps(z)-2\log(1/\eps)}dz\] where 
$T_\beta(z)=\sup\{\eps\leq d(z,\partial D): \Gamma_\eps(z)-2\log(1/\eps)\leq -\beta\}$. It is shown in \cite[Proposition 3.6]{EP} that $\nu_\eps^\beta(\Os)$ is uniformly integrable for fixed $\beta\geq 0$. Additionally, if we define \[C_\beta :=\{-\Gamma_\eps(z)+2\log(1/\eps)+\beta >0 \text{ for all }z\in \D,0<\eps\leq d(z,\partial \D) \},\]then $\P(C_\beta)=1-o(1)$ as $\beta \to \infty$ thanks to \cite[Theorem 6.15]{HRVdisk}.

The strategy is to prove that for \begin{equation}\label{deftau}\tau_\beta := \inf_n \left \{n\in \N:\inf_{z\in \D\backslash A_{an}}-h_{A_{an}}(z) + 2 \log \CR^{-1}(z,\D\setminus A_{an})\le -\beta\right \},
\end{equation}
the almost sure limit
\begin{equation}\label{eqn::key}\lim_{\beta\to \infty}\lim_{n\to \infty}\lim_{\eps\to 0} \E{\nu_\eps^\beta(\Os)\mid \F_{A_n}}\I{\tau_\beta=\infty}\end{equation} exists and is equal to both $\mu_2'(\Os)$ and $\frac{1}{2} \tilde D_\infty(\Os)$ almost surely.

\subsubsection*{Let us first show that (\ref{eqn::key}) is equal to $\mu'_2(\Os)$.} Observe that since $\nu_\eps^\beta(\Os)$ is uniformly integrable, we have by Fatou's and reverse Fatou's lemma that, if the limit in $\eps$ exists (we will show that it does in the next step)
\[ \E{\liminf_{\eps\to 0} \nu_\eps^\beta(\Os)\mid \F_{A_n}} \leq \lim_{\eps\to 0}\mathbb{E}[\nu_\eps^\beta(\Os)\mid \F_{A_n}]\leq \E{\limsup_{\eps\to 0} \nu_\eps^\beta(\Os)\mid \F_{A_n}}.\]
Taking the limit as $n, \beta \to \infty$ we obtain that
\[\lim_{\beta\to \infty}\liminf_{\eps\to 0} \nu_\eps^\beta(\Os) \leq \lim_{\beta\to \infty}\lim_{n\to \infty}\lim_{\eps\to 0}\mathbb{E}[\nu_\eps^\beta(\Os)\mid \F_{A_n}] \leq \lim_{\beta\to \infty}\limsup_{\eps\to 0} \nu_\eps^\beta(\Os) .\]
 However, since $\nu_\eps^\beta(\Os)=\nu_\eps(\Os)$ on the event $C_\beta$, and almost surely $\1_{C_\beta}\uparrow 1$ as $\beta\to \infty$, the right and left hand sides of the above two expressions are equal to $\mu'_2(\Os)$.
 
 Now, we only need to prove that almost surely $\I{\tau_\beta=\infty}\to 1$ as $\beta\to \infty$, which is due to the fact that 
 \begin{equation}\label{eqn::unifmin} 
 \inf_{z\in \D} \inf_{n\in \N} \left(-2an+4 \log \CR^{-1}(z, \D\setminus A_{an} )\right) > -\infty.
 \end{equation}
 This just follows from Lemma \ref{lem::criticaltozero} and its proof: indeed, as in the proof, one can observe that for any $z$ and any $n$, we have that $M^2_n( \D) \geq e^{2an-4 \log \CR^{-1}(z, \D\setminus A_{an})}$. Hence as $M^2_n(\D) \to 0$ by Lemma \ref{lem::criticaltozero}, we obtain \eqref{eqn::unifmin}. Thus,\eqref{eqn::key} is equal to $\mu_2'(\Os)$. 

\subsubsection*{We now show that (\ref{eqn::key}) is equal to $\frac{1}{2} \tilde D_\infty (\Os)$.} Similarly to the proof of Proposition \ref{propn::Ea} we write $\mathbb{E}[\nu_\eps^\beta(\Os)\mid \F_{A_n}]$ as the sum of:
\begin{align*} E^1(n,\beta,\eps) & := \int_{\Os\setminus A_n^\eps} \Ea{(-\Gamma_\eps(z)+2\log(1/\eps))\I{T_\beta(z)\leq \eps}\e^{2\Gamma_\eps(z)-2\log(1/\eps)}}dz , \\
E^2(n,\beta,\eps) & :=  \int_{\Os \cap A_n^\eps} \Ea{(-\Gamma_\eps(z)+2\log(1/\eps))\I{T_\beta(z)\leq \eps}\e^{2\Gamma_\eps(z)-2\log(1/\eps)}}dz \text{, and} \\
E^3(n,\beta,\eps) & := \beta \int_{\Os} \Ea{\I{T_\beta(z)\leq \eps}
		\I{\eps\le d(z,\partial D)}\e^{2\Gamma_\eps(z)-2\log(1/\eps)}}dz
\end{align*}
As before $A_n^\eps$ denotes the $\eps$-enlargement of $A_n$ and for shortness of notation we set $\mathbb E^{A_n}(\cdot) = \E{\cdot | \F_{A_n}}$. 

We first show that the terms $E^3(n,\beta,\eps)$ and $E^2(n,\beta,\eps)$ are negligible. For $E^3(n,\beta,\eps)$, we use the same calculation as in \eqref{eq:cond} to see that the limit in $\eps$ is less than or equal to $M_n$ which we know by Lemma \ref{lem::criticaltozero} converges to $0$ almost surely as $n\to \infty$. For $E^2(n,\beta,\eps)$ , on the one hand, by Definition \ref{def::fps} of $A_n$, we have that $(\Gamma_{A_n},\rho_z^\eps)\leq n$. On the other hand, for any $z \in A_n^\eps$ we have that conditionally on $A_n$, the variance of $\Gamma_\eps^{A_n}(z)$ is uniformly bounded. One way to see this is to write the variance explicitly using the Green's function, and to observe that the Green's function $G(z,w)$ is uniformly bounded for $d(w,z) \geq \eps/2$.
This implies that, for fixed $n,\beta$, the integrand of $E^2(n,\beta,\eps)$ is of order $\eps^2 \log(1/\eps)$ uniformly in $z$ and hence $E^2(n,\beta,\eps) \to 0$ as $\eps \to 0$.

We now deal with $E^1(n,\beta, \eps)$. Observe that if $\eps \leq d(A_n,z)$ then $(\Gamma_{A_n},\rho_z^\eps)=n$. Additionally, due to the Markov property of the GFF and an explicit computation, we have that conditionally on $\F_{A_n}$, i.e., under the probability $\P^{A_n}$,
\[\left (-n-\Gamma^{A_n}_\delta(z)+2\log(1/\delta)\right )\I{T_\beta(z)\leq \delta}\e^{2n+2\Gamma^{A_n}_\delta(z)-2\log(1/\delta)}\]
is a (reverse-time) martingale for $0<\delta \leq \delta_n(z):=d(z,\partial \D\cup A_n)$. 

Thus, we have that $E^1(n,\beta,\eps)$ is equal to
\[\int_{\Os\setminus A_n^\eps} \Ea{(-n-\Gamma^{A_n}_{\delta_n(z)}(z)+2\log(1/\delta_n(z)))\I{T_\beta(z)\leq \delta_n(z)}\e^{2n+2\Gamma^{A_n}_{\delta_n(z)}(z)-2\log(1/\delta_n(z))}}dz. \]
Since the integrand does not depend on $\eps$, the limit in $\eps$ exists almost surely and simply yields the integral over the whole of $\Os\setminus A_n$. 

Now, using that $\I{T_\beta(z)\leq \delta}= 1 - \I{T_\beta(z)>\delta}$, we rewrite $\lim_{\eps\to 0}E^1(n,\beta,\eps)$ as a difference between
\begin{equation}\label{eqn::e1_first}\int_{\Os\setminus A_n} \Ea{(-n-\Gamma^{A_n}_{\delta_n(z)}(z)+2\log(1/\delta_n(z)))\e^{2n+2\Gamma^{A_n}_{\delta_n(z)}(z)-2\log(1/\delta_n(z))}}dz \end{equation}
and 
\begin{equation}\label{eqn::e1_second}\int_{\Os\setminus A_n} \Ea{(-n-\Gamma^{A_n}_{\delta_n(z)}(z)+2\log(1/\delta_n(z)))\I{T_\beta(z)> \delta_n(z)}\e^{2n+2\Gamma^{A_n}_{\delta_n(z)}(z)-2\log(1/\delta_n(z))}}dz. \end{equation}
First notice that \eqref{eqn::e1_first} is equal to $D_n(\Os)/2$. This follows by a standard Gaussian calculation, as $\Gamma^{A_n}_{\delta_n(z)}$ is a mean zero normal random variable under $\mathbb{P}^{A_n}$ with variance $-\log \delta_n(z)+\log \CR(z,\D\setminus A_n).$ Therefore, we need only show that the random variable given by \eqref{eqn::e1_second} converges to $0$ on the event $\{\tau_{\beta}=\infty\}$, as $n\to \infty$ and then $\beta\to \infty$.

To show this, further decompose \eqref{eqn::e1_second} as a sum of
\begin{equation}\label{eqn::E4}
\int_{\Os\setminus A_n} \e^{2n-2\log(1/\delta_n(z))} \Ea{(-\Gamma^{A_n}_{\delta_n(z)}(z)+2\log\frac{\CR(z,\D\setminus A_n)}{\delta_n(z)})\I{T_\beta(z)> \delta_n(z)}\e^{2\Gamma^{A_n}_{\delta_n(z)}(z)}}dz.
\end{equation}
and
\begin{equation}\label{eqn::E3}
 \int_{\Os\setminus A_n} \e^{2n-2\log(1/\delta_n(z))}(-n+2\log\CR^{-1}(z,\D\setminus A_n))\Ea{\I{T_\beta(z)> \delta_n(z)}\e^{2\Gamma^{A_n}_{\delta_n(z)}(z)}}dz.
\end{equation}

Note that $\delta_n(z) \leq \CR(z, \D\setminus A_n) \leq 4\delta_n(z)$. This means that $\Gamma_{\delta_n(z)}^{A_n}(z)$ has bounded variance under $\mathbb{E}^{A_n}$, and so we see that \eqref{eqn::E4} is bounded by a constant times $M_n^2(\Os)$, which goes to $0$ by Lemma \ref{lem::criticaltozero}.

Finally, observe that on the event $\{\tau_\beta=\infty\}$, we have that
$-n+2\log\CR^{-1}(z,\D\setminus A_n)+\beta\geq 0$ for all $z \in \Os$.
Also, by Cauchy-Schwarz and the fact that $\Gamma_{\delta_n(z)}^{A_n}(z)$ has bounded variance under $\mathbb{E}^{A_n}$, we have
\[\Ea{\I{T_\beta(z)> \delta_n(z)}\e^{2\Gamma^{A_n}_{\delta_n(z)}(z)}} \leq c\P^{A_n}(C^c_\beta)^{1/2}.\]
This implies that on the event $\{\tau_\beta=\infty\}$ the absolute value of \eqref{eqn::E3} is upper bounded by
\[c'(|D_n(\Os)|/2+2\beta M_n(\Os)) \P^{A_n}(C^c_\beta)^{1/2}.\]
But the limit of the RHS as $n\to \infty$ is equal to $\frac{1}{2}\tilde D_\infty(\Os)\1_{C_{\beta-2}^c}$. Since this tends to $0$ as $\beta\to \infty$, we can conclude.	
\end{proofof}

\subsection{Seneta-Heyde rescaling.}\label{sectSH}

Finally, we show that one can also perform a so-called Seneta-Heyde rescaling to construct the critical Liouville measure using local sets. In fact we prove an even stronger result that will serve us in a follow-up paper, where we prove that the scaled subcritical measures $(2-\gamma)^{-1}\mu_\gamma$ converge to $2\mu_2'$. This result is known in the setting of multiplicative cascades \cite{Madaule}, and was conjectured in \cite{DSRV} for the Liouville measure. 

The moral of the proof can be summarised as follows: our set-up allows us to rather easily transfer the proofs from the multiplicative cascades setting to our context. In particular, the proof in this section follows very closely the proof of \cite{AiSh}, and its extension in \cite{Madaule}. Not only is the set-up of the proof exactly the same, but also technical details can be easily translated to our setting. We have, however, aimed to make this section self-contained and have simplified and shortened some of the technical steps. 

The main result of this section is the so called Seneta-Heyde rescaling. Let $M_n^2$ be defined as in Proposition \ref{propn::Ea} (with $a=1$ and $\gamma =2$), i.e. \[M_n^2(dz) = \e^{2n-2\lzn}dz.\]
We then see that there is a suitable rescaling of $M_n^2$ that converges to the derivative measure $\mu_2'$:

\begin{thm}[Seneta-Heyde Rescaling]\label{sh} \label{prop::seneta_hyde} For all Borelian $\Os \subseteq \D$ (including $\Os=\D$), 
	$\sqrt{n}M^2(\Os,A_{n})$ converges to $\frac{2}{\sqrt{\pi}}\mu_2'(\Os)$
	in probability as $n\to \infty$. In particular, we have $$\sqrt{n}M^2_n\to \frac{2}{\sqrt{\pi}}\mu_2'$$ weakly in probability as $n\to \infty$.
\end{thm}

In fact this is a direct consequence of a stronger and more general statement, that will serve us in a follow-up work. First let \[D_n(\Os):=2D^2(\Os,A_n)\] as in Section \ref{sec::derivativemgale}.

\begin{thm} \label{prop::shstrong}
	Suppose that $F:\R \to \R^+$ is a positive, continuous and bounded function, and let 
	\[ K_n^F(dz):= \e^{2n-2\lzn} F\left(\frac{-2n+4\lzn}{\sqrt{n}}\right) \, dz.\]
	Then for any Borelian $\Os \subseteq \D$ (inlcuding $\Os=\D$) we have
	\begin{equation*}
	\frac{\sqrt{n} K_n^F(\Os)}{D_n(\Os)} \to \sqrt{\frac{1}{\pi}} \mathbb{E}[F(\sqrt{2}R_1)]
	\end{equation*}
	in probability as $n\to \infty$, where $R_1$ has the law of a Brownian meander at time $1$. 
\end{thm}

Indeed, in order to deduce Theorem \ref{sh} from this general statement we first take $F = 1$ to conclude that $\sqrt{n} M_n^2(\Os)/D_n(\Os)\to \sqrt{1/\pi}$ in probability, as $n\to \infty$.  As by Proposition \ref{propn::derivative} we also have that $D_n(\Os)\to 2\mu'(\Os)$ almost surely, Theorem \ref{sh} follows by invoking Remark \ref{IdWC}.

It is convenient to work (we will later explain why) under a certain family of \emph{rooted measures}, that heuristically amount to picking a typical point from the critical measure.

\subsubsection{Another family of rooted measures}\label{sec:rm}

 Recall that in Section \ref{sec:crz} we already made use of certain rooted measures, one for each value of $\gamma$, obtained by weighting our original measure by the martingale $M_n^\gamma$. 
 
 To prove Theorem \ref{prop::shstrong}, we will define a different family $\mathbb{Q}_\eta$ of rooted measures, using certain martingales $(\tilde{D}_n^\eta)_{\eta \ge 0}$ (defined below) as a weighting instead. These martingales provide truncated approximations of the critical measure, with $\eta > 0$ the truncation parameter. 
 
 We already saw in the proof of Lemma \ref{lem::criticaltozero} that if $(\Gamma,Z)$ has the law $\hat{\mathbb{P}}^*(d\Gamma, dz)$ defined in \eqref{Pstar}, then the process 
 \[S_n(Z):=-2n+4\log\CR^{-1}(Z,\D\setminus A_n)\] 
 is a random walk with mean-zero increments under the conditional law $\hat{\mathbb{P}}^*(d\Gamma |Z)$. Moreover, this conditional law is that same as that of $S_n$ under $\mathbb{E}$, but weighted by 
 \[ \e^{2n-2\log \CR^{-1}(Z,\D\setminus A_n)+2\log \CR^{-1}(Z,\D)}.\] 
 By conformal invariance of the GFF, this implies that the conditional law of $(S_n-S_0)$ under $\hat{\mathbb{P}}^*(d\Gamma |Z)$ does not depend on $Z$ (although $S_0=4\log \CR^{-1}(Z,\D)$ clearly does.) 

Now, let us define, as in \cite{Ai},
\[\tilde{D}_n^\eta(\Os) := \int_{\Os} h_1\left(S_n(z)+2\eta\right)1_{E_\eta(n,z)}e^{2n}\CR(z; \D\setminus A_n)^2\, dz,\]
where $E_\eta(n,z):= \{S_m \ge -2\eta \text{ for all } m\leq n\}$,
and $h_1$ is the renewal function associated with the random walk $(S_n-S_0)$ under $\hat{ \mathbb P}^*(d\Gamma \, | \, Z)$:
\begin{equation}
\label{eqn::renew_defn}
h_1(u):= \hat{ \mathbb P}^*\left(\left. \sum_{j=0}^{\infty} \I{\inf_{i\leq j-1} (S_i - S_0) > S_j-S_0\geq -u} \, \right| \, Z \right) \ge 1, \ \ u\geq 0.
\end{equation}
This is a deterministic function of $u$ (in particular, not depending on $Z$) by the discussion above. We have collected further background and properties of the renewal function in Appendix \ref{sec::appendix}.

Proposition 3.2 of \cite{Ai} implies that for all $\eta > 0$, $\tilde{D}_n^\eta(\Os)$ is a uniformly integrable positive martingale with respect to $(\F_{A_n})_n$ and our initial probability measure $\mathbb{P}$. Hence, we can define a new probability measure
$\mathbb{Q}_{\eta}$ by setting 

\begin{equation}\label{eqn::q_eta_com} \frac{d\mathbb{Q}_{\eta}}{d\mathbb{P}} \big|_{\F_{A_n}} :=  \frac{\dne(\Os)}{\tilde{D}_0^\eta(\Os)}.\end{equation}
\footnote{Note that, by definition, the measure $\mathbb{Q}_\eta$ also depends on the set $\Os$. This is just a technical convenience, and we omit the dependence from the notation, as it should always be contextually clear which $\Os$ we are using.} Again we extend this to a \emph{rooted} measure on the field $\Gamma$ plus a distinguished point $Z$ by setting  $\mathbb{Q}_\eta^*(d\Gamma, dz)$ restricted to ${\F_{A_n}^*} = \F_{A_n}\vee \sigma(Z)$ to be equal to 
\[h_1(S_n(z) +2\eta)\e^{2n-2 \log \CR^{-1}(z,\D\setminus A_n)}\ind{z}{n}\frac{\1_{\Os}(z)}{\tilde{D}_0^\eta} \,dz \,\mathbb{P}[d\Gamma].\]

We make the following observations, which follow from direct calculations, together with the Markov property of the renewal function \eqref{eqn::h_1_renewal}:
\begin{enumerate}
	\item The marginal law of $Z$ under $\mathbb{Q}_\eta^*$ is proportional to $h_1(4\log \CR^{-1}(z,\D)+2\eta)\CR(z,\D)^{2} \1_{\Os}(z) dz$.
	\item The marginal law of the field $\Gamma$ under $\mathbb{Q}_\eta^*$ is given by $\mathbb{Q}_\eta$.
	\item The conditional law of $Z$ given the field $\Gamma$ has density $$ \tilde{D}^\eta_n(\Os)^{-1} h_1(S_n(z)+2\eta)\e^{2n-2\log \CR^{-1}(z,\D\setminus A_n)}\1_{E_\eta(n,z)} \I{z\in \Os} $$ with respect to Lebesgue measure.
	\item Finally, write  $\mathbb{Q}_{\eta,z}^*=\mathbb{Q}_\eta^*[\cdot \mid Z = z]$ for the law of $\Gamma$ given the point $Z = z$. The law of the sequence $(A_n)_n$ under this measure can be described as follows. First sample $A_1$ with law weighted by 
	\[ \frac{h_1(S_1(z)+2\eta)}{h_1(S_0(z)+2\eta)} \ind{z}{1}\e^{2-2\log \CR^{-1}(z,\D\setminus A_1)+2\log \CR^{-1}(z,\D)}.\]
	Then, given $A_k$ for any $k\geq 1$, construct an independent copy of $(A_n)_n$ inside each component of $\D\setminus A_k$ that does \emph{not} contain the point $z$. Inside the component containing $z$, let us call this $\mathcal{B}_k$, construct the components of $A_{k+1}\cap \mathcal{B}_k$ by weighting their laws by
	\[ \frac{h_1(S_{k+1}(z)+2\eta)}{h_1(S_k(z)+2\eta)} \ind{z}{k+1} \e^{2-2\log \CR^{-1}(z,\D\setminus A_{k+1})+2\log \CR^{-1}(z,\D\setminus A_k)}. \] This defines the law of the sets $A_n$, and hence by iteration, the law of $\Gamma$.
\end{enumerate}

It then follows that for any $n$ the law of $(S_k(z))_{1\leq k \leq n}:=(-2k+4\log \CR(z,\D\setminus A_i))_{1\leq k \leq n}$ under $\qex$ is the same as its law under $\hat{\mathbb{P}}^*[\cdot \mid Z=z]$, but weighted by 
	\begin{equation}
	\label{eqn::doob_weight}
	 \frac{h_1(S_n(z) + 2\eta)}{h_1(S_0(z)+2\eta)} \1_{E_\eta(z,n)}.\end{equation} 
By the classical theory of Doob h-transforms this weighting is the same as conditioning $(S_n(z))_n$ to stay above $-2\eta$ (see for example \cite[Fact 3.2 (iii)]{AiSh}).

\subsubsection{Proof of Theorem \ref{prop::shstrong}}

In order to prove Theorem  \ref{prop::shstrong} we would like to use a first and second moment method to show that the random variables $\sqrt{n} K_n^F(\Os)/D_n(\Os)$ converge to a constant. However, these moments might explode a priori. Thus, we truncate the random variables: turning $D_n$ into $\tilde D_n^\eta$ and also adding the indicator $\I{E_\eta(n,z)}$ in the definition of $K_n^F$. Once we have done this,  it actually turns out to be more convenient to work under the rooted measure in order to study the truncated ratio. This is partly because, under the rooted measure, the ratio can be naturally written as a functional of the marked point $Z$.

So, we set
\[\tilde{K}_n^{F,\eta}(\Os):= \int_{\Os}e^{2n-2\lzn}F\left (\frac{S_n(z)}{\sqrt{n}}\right )\ind{z}{n} \, dz.\] 

As mentioned above, the proof of Theorem  \ref{prop::shstrong} follows by studying the behaviour of $\sqrt{n}\tilde{K}_n^{F,\eta}/\tilde{D}_n^\eta$ under the rooted measure $\mathbb{Q}_\eta$. More precisely, we establish the following proposition:
\begin{propn}\label{claim::sh} For any $\eta>0$ and all Borel $\Os\subseteq \D$ (including $\Os=\D$)
	\label{lem::convqaprob} \[\sqrt{n} \frac{\tilde{K}^{F,\eta}_n(\Os)}{\tilde D_n^{\eta}(\Os)}  \to\frac{1}{c_0\sqrt{\pi}}  \E{F(\sqrt{2}R_1)}\] in $\mathbb{Q}_{\eta}$-probability as $n\to \infty$, where 
 $c_0\in (0,\infty)$ is such that $
h_1(u)/u \to c_0 \text{ as } u\to \infty$, see \eqref{eqn::renewal_thm}, and $R_1$ is a Brownian meander at time $1$.
\end{propn}

Before proving this proposition, let us shortly explain how it implies the theorem. Let us stress once again that we have set things up so that we can very closely follow the proof of Theorem 1.1 in \cite{AiSh}.
\medskip 

\begin{proofof}{Theorem \ref{prop::shstrong} assuming Proposition \ref{claim::sh}}
To see heuristically why this proposition suffices, observe that thanks to \eqref{eqn::unifmin}, almost surely there is a (random) $\eta_0$ such that whenever $\eta\ge \eta_0$, the event $E_\eta(z,n)$ holds for all $n\in \N$ and $z\in \D$. This means that for all $\eta\ge \eta_0$, we have \[\tilde{K}_n^{F,\eta}(\Os)=K_n^F(\Os) \text{ and } \tilde{D}_n^\eta(\Os) \overset{n\to \infty}{\sim} c_0 D_n(\Os).\]
Moreover, we have seen that $\tilde{D}_n^\eta(\Os)$ converges almost surely to a strictly positive limit as $n\to \infty$ and then $\eta \to \infty$. It follows that convergence in probability under $\mathbb{P}$ and under $\mathbb{Q}_\eta$ are comparable when $\eta$ is large, implying Theorem \ref{prop::shstrong}. For an interested reader, the details are given in Appendix \ref{sec::appendixB}.
		\end{proofof}
\\
	
To prove Proposition \ref{claim::sh}, we first treat the case when $\Os$ is compactly supported in $\D$ and then in the end discuss how to extend this to sets intersecting the boundary of $\D$. This is to separate certain technicalities arising when working near the boundary.
\medskip

\begin{proofof}{Proposition \ref{claim::sh} for $\Os$ compactly supported in $\D$} We may assume without loss of generality that $\Os \subset r\D$ for some $r < 1$. From now on, we also omit the argument $\Os$ in $\tilde{K}_n^{F,\eta}(\Os)$ and $\tilde{D}_n^\eta(\Os)$ etc., in order to keep notations compact. 
	
	Define $\theta=1/(c_0\sqrt \pi)$.
	The idea is to control the first and second moments of $ \tilde{K}_n^{F,\eta}/\bar D_n^\eta$ as $n \rightarrow \infty$. More precisely, to show that: 
	\begin{align}
	\label{eqn::firstmoment}
	&\qb{\frac{\tilde K_n^{F,\eta}}{\tilde D_n^\eta}}= \frac{\theta \E{F(\sqrt{2}R_1)}}{\sqrt{n}} + o\left(\frac{1}{\sqrt{n}}\right) \text{ ; and}\\
\label{eqn::secondmoment}
&\qb{\left(\frac{\tilde K_n^{F,\eta}}{\tilde D_n^\eta}\right)^2} \leq \frac{(\theta\E{F(\sqrt{2}R_1)})^2}{n}+o\left(\frac{1}{n}\right)
	\end{align}
	as $n\to \infty$. Note that we have written $\mathbb{Q}_\eta^*$ rather than  $\mathbb{Q}_\eta$ here, but by observation (2) above, this makes no difference to the expectations (since the random variables inside are measurable with respect to $\mathcal{F}_{A_n}$). These estimates then prove Lemma \ref{claim::sh}, as they show that the variance of $\sqrt{n}\tilde{K}_n^{F,\eta}/\tilde{D}_n^\eta$ converges to $0$ with $n$.

 The key observation for the proofs of \eqref{eqn::firstmoment} and \eqref{eqn::secondmoment} lies in rewriting the moments using the rooted measure. Indeed, by observation (3) above, we can write:
	\begin{equation}
	\label{eqn::conditionalexp}
	\frac{\tilde{K}_n^{F,\eta}}{\tilde{D}_n^\eta} = \qb{ \frac{F(S_n(Z)/\sqrt{n})}{h_1(S_n(Z)+2\eta)} \mid \F_{A_n}}. \end{equation} 
	The advantage of this is that the functionals inside the expectation are just real-valued functions. Moreover, we know precisely the distribution of $Z$ under $\mathbb{Q}_\eta^*$. This allows us to directly calculate the first moment, and to control the second moment.

\subsubsection*{First moment estimate}
From the previous equation it follows that \begin{equation}
	\label{eqn::conditionalexp2}\mathbb{Q}_{\eta}^*\left[\frac{\tilde K_n^{F,\eta}}{\tilde D_n^\eta}\right]=\int_z \qa{\frac{F(S_n(z)/\sqrt{n})}{h_1(S_n(z)+2\eta)}} \mathbb{Q}_\eta^*[dz]\end{equation}
	 where $\mathbb{Q}_{\eta,z}^*$ represents the conditional law $\mathbb{Q}_\eta^*(\cdot | \{Z=z\})$, and $\mathbb{Q}_\eta^*[dz]$ the marginal density of $Z$ under $\mathbb{Q}_\eta^*$ as in Section \ref{sec:rm}.
	 
	 Now, by \eqref{eqn::doob_weight} we have for any $z\in \Os$ that
	$$ \qa{\frac{F(S_n(z)/\sqrt{n})}{h_1(S_n(z)+2\eta)}}=\frac{\hat{\mathbb{P}}^*\left( F(S_n(z)/\sqrt{n})\1_{E_\eta(n,z)}  \mid  \{Z=z\}\right)}{h_1(4\log \CR^{-1}(z,\D)+2\eta)}.$$ 
Applying \eqref{eqn::persistency} we see that this is equal to $\theta(1+o(1))/\sqrt{n}$ times
\begin{equation*} \hat{\mathbb{P}}_z^*\left( F\left(\frac{S_n(z)-S_0(z)+4\lzo}{\sqrt{n}}\right) \, \big| \, E_\eta(n,z)  \right)\end{equation*} 
Here we have written $\hat{ \mathbb P}_z^*$ for the law $\hat{ \mathbb P}^*(\cdot \mid \{Z=z\})$, but remember that the law of $S_n(z)-S_0(z)$ under $\hat{ \mathbb P}_z^*$ does not actually depend on $z$: it is a random walk with mean zero increments and variance equal to $2$ (see Remark \ref{rmk::variance}). This is now an expression we can deal with easily, because it is well known \cite{Bolthausen} that a mean zero bounded variance random walk conditioned to stay above some level converges to a Brownian meander. Thus
for every $z \in \Os$ and $\eta > 0$: \begin{equation}\label{eqn::donsker_law} n^{-1/2}(S_{\lfloor t/n\rfloor }(z)-S_0(z))_{0\leq t \le n}  \text{ under }  \hat{ \mathbb P}^*_z\left(\cdot  \mid E_\eta(n,z)\right) \, \end{equation} converges in distribution to $(\sqrt{2}R_t)_{0\le t \le 1}$ as $n\to \infty$, where $R$ is a Brownian meander on $[0,1]$. Hence, using continuity and boundedness of $F$, the integrand of the right-hand side of \eqref{eqn::conditionalexp2} is equal to $\theta(1+o(1))\E{F(\sqrt{2}R_1)}/\sqrt n$ for every $z\in \Os$. Note that the $o(1)$ is uniform over $z\in \Os$ since $S_0(z)=4\log \CR^{-1}(z,\D)$ is uniformly bounded over $z\in r\D$. By dominated convergence, we therefore obtain \eqref{eqn::firstmoment}. 

\subsubsection*{Second moment estimate}
We now move to	the second moment estimate. The idea is as follows. Using \eqref{eqn::conditionalexp} we can write 
\begin{equation}\label{eqn::K/D2}\qb{\left(\frac{\tilde K_n^{F,\eta}}{\tilde D_n^\eta}\right)^2}=\qb{\frac{\tilde K_n^{F,\eta}}{\tilde D_n^\eta} \frac{F(S_n(Z)/\sqrt{n})}{h_1(S_n(Z)+2\eta)}}
\end{equation}
Intuitively, the ratio $\tilde K_n^{F,\eta}/\tilde D_n^\eta$ will not depend too much on the final iterations of $A_n$, i.e. it depends only on $A_m$ for $m \leq k_n$ for some $k_n \ll n$. On the other hand the random walk $S_n(Z)$, doesn't depend too much on the initial iterations of $A_n$, i.e. on $A_m$ for $m \leq k'_n$ for $k'_n \ll n$. Thus, one may hope to argue that both terms in the right-hand side of \eqref{eqn::K/D2} are asymptotically independent. Making this precise and showing that in fact one can really take $k_n = k_n'$ is the content of the rest of the section. It will require quite a few technical steps, but luckily for us, these can be transferred with minor modifications from \cite{AiSh}. 

The first step is to get a rough upper bound of the right order:
	\begin{lemma}\label{lem::secondmoment_rough}
		\begin{equation}
		\label{eqn::secondmoment_rough}
		\qb{\left( \frac{\tilde K_n^{F,\eta}}{\tilde D_n^\eta}\right)^2}=O\left (\frac{1}{n}\right ).
		\end{equation} 
	\end{lemma} 

\begin{proofof}{Lemma \ref{lem::secondmoment_rough}}
Using (\ref{eqn::conditionalexp}) and Jensen's inequality, one sees that
	\begin{align}\label{eqn::condsqexp1}\qb{\left( \frac{\tilde K_n^{F,\eta}}{\tilde D_n^\eta}\right)^2}
	&	\le \|F\|_\infty^2 \qb{ \frac{1}{[h_1(S_n(Z)+2\eta)]^2}} \nonumber \\
	&	\leq \|F\|_\infty^2 \int_z \qa{\frac{1}{[h_1(S_n(z)+2\eta)]^2}} \qb{dz}.\end{align}
	Then, using the fact that $h_1(S_n(z)+2\eta)\geq R(1+S_n(z)+2\eta)$ by \eqref{eqn::h_1_bounds}, we can write 
	\[\qa{\frac{1}{[h_1(S_n(z)+2\eta)]^2}} \le  R^{-2} \hat{ \mathbb P}^*_z(\frac{1}{1+S_n(z)+2\eta}\mid E_\eta(n,z))\, \hat{ \mathbb P}^*_z(E_\eta(n,z))\] for all $z$. 
Applying \eqref{eqn::persistency}, as in the proof of \eqref{eqn::firstmoment}, we see that \[\hat{ \mathbb P}^*_z(E_\eta(n,z)) =\frac{\theta(1+o(1))}{\sqrt{n}}\] where again, since $\Os \subset r\D$, the $o(1)$ term is uniform over $z\in \Os$. By the convergence of $n^{-1/2}(S_{\lfloor t/n\rfloor}(z) - S_0(z))$ to a Brownian meander, \eqref{eqn::donsker_law}, we also see that the first term is order $n^{-1/2}$, uniformly over $z\in \Os$.
Plugging into \eqref{eqn::condsqexp1} and applying dominated convergence, we obtain the lemma.
\end{proofof}
\medskip 

This means that we need only prove the second moment bound on events of high probability. More precisely: 
	\begin{lemma}
		\label{lemma::En_sufficient}
	Suppose we can find a sequence of events $\Ev_n$ with $\mathbb{Q}_\eta^*(\Ev_n)\to 1$ as $n\to \infty$, and such that 
	\begin{equation}\label{eqn::condsqexp}\qb{\frac{\tilde K_n^{F,\eta}}{\tilde D_n^\eta} \frac{\1_{\Ev_n} F(S_n(Z)/\sqrt{n})}{h_1(S_n(Z)+2\eta)}} \leq \frac{(\theta\E{F(\sqrt{2}R_1)})^2}{n}+o\left(\frac{1}{n}\right).\end{equation} 
 Then the second moment bound \eqref{eqn::secondmoment} holds.
	\end{lemma} 
The proof of this lemma is also relatively direct: 
\medskip

\begin{proofof}{Lemma \ref{lemma::En_sufficient}}
By (\ref{eqn::conditionalexp}) we can write 
	\begin{align*}\qb{\left(\frac{\tilde{K}_n^{F,\eta}}{\tilde{D}_n^{\eta}}\right)^2} =\qb{\frac{\tilde{K}_n^{F,\eta}}{\tilde{D}_n^\eta}\frac{\1_{\Ev_n}F(S_n(Z)/\sqrt{n})}{h_1(S_n(Z)+2\eta)}}+\qb{\frac{\tilde{K}_n^{F,\eta}}{\tilde{D}_n^\eta}\frac{\1_{\Ev_n^c}F(S_n(Z)/\sqrt{n})}{h_1(S_n(Z)+2\eta)}}.\end{align*}
	As by assumption the first term is smaller or equal than $ n^{-1}(\theta\E{F(\sqrt{2}R_1)})^2+o\left(1/n\right)$, it suffices to show that the second term above is $o(1/n)$ as $n\to \infty$. By Cauchy--Schwarz, Lemma \ref{lem::secondmoment_rough} and boundedness of $F$, it is enough to show the same for \begin{equation}\label{eqn::e_n_suff_def}\qb{\1_{\Ev_n^c}(h_1(S_n(Z)+2\eta)^{-2}}.\end{equation} 
	For this, consider the event $F_\eps = \{S_n - S_0 \geq \eps n^{1/2}\}$ and write \eqref{eqn::e_n_suff_def}
as the sum
\[ \qb{\I{F_\eps}\I{\Ev_n^c}(h_1(S_n(Z)+2\eta)^{-2}}  + \qb{\I{F_\eps^c}\I{\Ev_n^c}(h_1(S_n(Z)+2\eta)^{-2}}. \]
	Using that $h_1(u) \geq R(1+u)$ and that $S_0(z)+2\eta\ge 0$ under $\mathbb{Q}_{\eta,z}^*$, we can bound the first term above by $c(1+\eps^2n)^{-1}\qb{\I{\Ev_n^c}}$. Moreover, as in Lemma \ref{lem::secondmoment_rough} we can now use the convergence to the Brownian meander and the fact that $z \in r\D$ to bound the second term by $n^{-1}c'\eps$ for some absolute constant $c'$. As $\eps$ can be taken arbitrarily small, the lemma follows.

\end{proofof}
\\

It thus remains to find a suitable sequence of events $\Ev_n$: such that $\mathbb{Q}_\eta^*(\Ev_n)\to 1$ and that \eqref{eqn::condsqexp} holds. As hinted at in the heuristic discussion before Lemma \ref{lem::secondmoment_rough}, the idea is to find an appropriate separation of scales. We will show that $k_n = k_n' = n^{1/3}$ is the right choice.

Indeed, pick $k_n = n^{1/3}$ and decompose $\tilde K_n^{F,\eta}$ and $\tilde D_n^\eta$ by writing \[\tilde K_n^{F,\eta}=\tilde K_n^{F,\eta,k_n^-}+\tilde K_n^{F,\eta, k_n^+} \text{ ; and } \tilde D_n^\eta = \tilde D_n^{\eta, k_n^+}+\tilde D_n^{\eta, k_n^-}\]
	where the superscript $k_n^+$ refers to the integral over $\mathcal{B}_{k_n}$ and the superscript $k_n^-$ refers to the integral over $\Os\setminus \mathcal{B}_{k_n}$ (recall that $\mathcal{B}_{k_n}$ was defined to be the connected component of $\D\setminus A_n$ containing $Z$).
	
	We now define our sequence of events $\Ev_n$ by setting $\Ev_n=\Ev_n^1\cap \Ev_n^2$, where
	\[ \Ev_n^1 := \{\tilde D_n^{\eta,k_n^+}\leq 1/n^2\}; \;\; \Ev_n^2 = \{S_{k_n}(Z)\in [k_n^{1/3},k_n]\}. \] 
	Under $\mathbb{Q}_{\eta,z}^*$, the walk $S_n(z)$ is a random walk with finite variance increments, conditioned to stay above $-2\eta$. It then follows that it will with high probability stay inside a window $[n^{1/3},n]$ for $n$ large enough. Hence we obtain:	
\begin{lemma}\label{lemma::e_n_2} 
	$\qb{\Ev_n^2} \to 1 $ as $n\to \infty$.
	\end{lemma}

We next claim that conditioned on $\Ev_n^2$ the event $\Ev_n^1$ also happens with large probability.
		\begin{lemma}\label{lemma::e_n} There exists a deterministic sequence $p_n\nearrow 1$ such that $\1_{\Ev_n^2} \qb{ \Ev_n^1 \mid \F^*_{A_{k_n}}} \geq p_n\1_{\Ev_n^2} $ almost surely. Here  $\mathcal{F}_{A_{k_n}}^*$ is the $\sigma$-algebra generated by $\mathcal{F}_{A_{k_n}}$ together with $\sigma(Z)$
	\end{lemma}
We will postpone the proof of this lemma. For now, observe that since $\Ev_n^2$ is $\mathcal{F}_{A_{k_n}}^*$ measurable, the combination of Lemmas \ref{lemma::e_n_2} and \ref{lemma::e_n} imply that $\qb{\Ev_n}\to 1$ as $n\to \infty$ . 

Hence it remains to  prove \eqref{eqn::condsqexp}. Using positivity of the integrands defining $\tilde{K}_n^{F,\eta}$ and $\tilde{D}_n^\eta$, the first step is to bound the left hand-side of \eqref{eqn::condsqexp} by 
	\begin{equation}\label{eqn::decomp_secondmomentbound}\qb{\frac{\tilde K_n^{F,\eta, k_n^+}}{\tilde D_n^\eta}\frac{\1_{\Ev_n^1}F(S_n(Z)/\sqrt{n})}{h_1(S_n(Z)+2\eta)}} + \qb{\frac{\tilde K_n^{F,\eta,k_n^-}}{\tilde D_n^{\eta,k_n^-}}\frac{\1_{\Ev_n}F(S_n(Z)/\sqrt{n})}{h_1(S_n(Z)+2\eta)}}.\end{equation}
Next, on the event $\Ev_n^1$, we have that \[0\le \tilde K_n^{F,\eta, k_n^+} \leq  \|F\|_\infty\tilde D_n^{\eta, k_n^+} \leq \|F\|_\infty/n^2.\]
Using the definition of $\mathbb Q^*_\eta$ and bounding $h_1 \geq 1$ we see that the first term of \eqref{eqn::decomp_secondmomentbound} is smaller than $n^{-2}$ times $\|F\|_\infty^2/D_0^\eta(\Os)$. Since we also have $D_0^\eta(\Os) > 0$, the first term is of order $o(1/n)$.

For the second term, we use that the two ratios in the expectation are conditionally independent given $\F^*_{A_{k_n}}$. This means that we can write 
	\begin{align}
	\qb{\frac{\tilde K_n^{F,\eta,k_n^-}}{\tilde D_n^{\eta,k_n^-}}\frac{\1_{\Ev_n}F(S_n(Z)/\sqrt{n})}{h_1(S_n(Z)+2\eta)} \mid \F^*_{A_{k_n}}} \nonumber 
	\end{align}
	as
	\begin{align}
	 \qb{\frac{\tilde K_n^{F,\eta,k_n^-}}{\tilde D_n^{\eta,k_n^-}} \mid \F^*_{A_{k_n}}} \1_{\Ev_n^1} \1_{\Ev_n^2} \qb{  \frac{F(S_n(Z)/\sqrt{n})}{h_1(S_n(Z)+2\eta)} \mid \F^*_{A_{k_n}}}. 
	\end{align} We then have, by (the comment following) \eqref{eqn::persistency}, and the proof of \eqref{eqn::firstmoment}, that \[\1_{\Ev_n^2}\qb{\frac{F(S_n(Z)/\sqrt{n})}{h_1(S_n(Z)+2\eta)}\mid \F^*_{A_{k_n}}}=\left( \frac{\theta\E{F(\sqrt{2}R_1)}}{\sqrt{n}}+o(1/\sqrt{n})\right)\1_{\Ev_n^2},\] where the $o(1/\sqrt{n})$ is deterministic. It therefore remains to prove that \begin{equation}
	\label{eqn::firstcondexp}
	\qb{\frac{\tilde{K}_n^{F,\eta,k_n^-}}{\tilde{D}_n^{\eta,k_n^-}} \1_{\Ev_n}}\leq \frac{\theta\E{F(\sqrt{2}R_1)}}{\sqrt{n}}+o(1/\sqrt{n}).\end{equation}
	To do this, we break  up (\ref{eqn::firstcondexp}) as
	\[\qb{\frac{\tilde{K}_n^{F,\eta,k_n^-}}{\tilde{D}_n^{\eta,k_n^-}} \1_{\Ev_n}\I{\tilde{D}_n^\eta > 1/n}}+ \qb{\frac{\tilde{K}_n^{F,\eta,k_n^-}}{\tilde{D}_n^{\eta,k_n^-}} \1_{\Ev_n}\I{\tilde{D}_n^\eta \le 1/n}}.\]
	Note that $\tilde{K}_n^{F,\eta,k_n^-}$ is smaller than $\|F\|_\infty\tilde{D}_n^{\eta,k_n^-}$. It therefore follows that we can bound the second term by $(\tilde{D}_0^\eta)^{-1} \E{\tilde{D}_n^\eta \I{\tilde{D}_n^\eta \le 1/n}}$, which is again  $o(1/\sqrt{n})$ since $\tilde D_0^\eta$ is non-zero.
	
	Moreover, on the event $\Ev_n \cap \{\tilde{D}_n^\eta >1/n\}$ we have $\tilde{D}_n^\eta/\tilde{D}_n^{\eta,k_n^-}=1+O(1/n)$, and so we see that the first term of (\ref{eqn::firstcondexp}) is less than or equal to $(1+O(1))$ times the first moment in \eqref{eqn::firstmoment}. Since we already know that this is $\theta \E{F(\sqrt{2}R_1)}/\sqrt{n}+o(1/\sqrt{n})$, the proof is complete. 
\end{proofof}
\medskip

It is only in the proof of the final lemma, Lemma \ref{lemma::e_n}, that we need to do a bit of extra work over that already done in \cite{AiSh}. This comes from the fact that, unlike in the classical setting of multiplicative cascades, the sets $A_n$ at the $n-$th level have different shapes and sizes. In this lemma, we can however use the work of \cite{Ai}.
	\medskip
	
\begin{proofof}{Lemma \ref{lemma::e_n}} Define further events $\Ev_n^3$ and $\Ev_n^4$ by setting
	\[\Ev_n^3 =  \cap_{k_n\leq j\leq n} \{S_j(Z)\geq k_n^{1/6}\}; \;\; \Ev_n^4=\cap_{k_n\leq j\leq n} \{\text{sup}_{w\in \mathcal{B}_j} |Z-w|\leq j^{c}\CR(Z,\mathcal{B}_j)\}\]
	where $\mathcal B_j$ is the connected component of $D\backslash A_{j}$ containing $Z$ and $c>0$ is some fixed constant to be chosen just below (see also \cite[Lemma 3.5]{Ai}). We argue that: 
	\begin{enumerate}[(i)]
		\item $\1_{\Ev_n^2}\qb{ \Ev_n^3 \mid \F^*_{A_{k_n}}} \geq p_n\1_{\Ev_n^2}$, where $p_n\to 1$ is deterministic;
		\item $\1_{\Ev_n^2}\qb{ \Ev_n^4 \mid \F^*_{A_{k_n}}}\geq q_n\1_{\Ev_n^2}$, where $q_n\to 1$ is deterministic; and finally
		\item $\qb{\tilde D_n^{\eta, k_n^+}\1_{\Ev_n^3\cap \Ev_n^4} \mid \F^*_{A_{k_n}}}\leq r_n$ where $r_n=o(1/n^2)$ is deterministic.
	\end{enumerate}
	This proves the lemma by conditional Markov's inequality.
	
	Statement (i) follows from the fact that under the given conditional law, $(S_j(Z)-S_{k_n}(Z); \; j\geq k_n)$ is a centered random walk conditioned to stay above $-S_{k_n}(Z)+2\eta$. 
	
	Statement (ii) follows from the proof of \cite[Lemma 3.5]{Ai}. This proof shows that for $c$ large enough, 
	\[\qb{\text{sup}_{w\in \mathcal{B}_j} |Z-w| > j^c \CR(Z,\mathcal{B}_j) \mid \F^*_{A_{k_n}}}\leq c'j^{-2},\]
	for some constant $c'$ (note the right-hand side is deterministic.)  
	
For (iii), we bound
		$\qb{\tilde D_n^{\eta, k_n^+}\1_{\Ev_n^3\cap \Ev_n^4} \mid \F^*_{A_{k_n}}} $ above by 
		\begin{align}\label{eqn::brotherloops}
		 \qb{\1_{\Ev_n^3\cap \Ev_n^4}  \int_{\mathcal{B}_n} h_1(-2n+4\log \CR^{-1}(w,\mathcal{B}_n)+2\eta)\1_{E_{\eta}(w,n)}\e^{2n}\CR(w,\mathcal{B}_n)^2\, dw \mid \F^*_{A_{k_n}}}  & \\ \nonumber
		  + \sum_{j=k_n}^{n-1} \qb{  \1_{\Ev_n^3\cap \Ev_n^4} \int_{\mathcal{B}_j\setminus\mathcal{B}_{j+1}} h_1(-2n+4\log \CR^{-1}(w,A_n)+2\eta)\1_{E_{\eta}(w,n)}\e^{2n}\CR(w,A_n)^2 \mid \F^*_{A_{k_n}}} &
		\end{align}

To control each term on the second line of \eqref{eqn::brotherloops}, we condition further on all the \emph{brother loops} of the point $Z$ at level $(j+1)$; that is, the components of $\D\setminus A_{j+1}$ contained in $\mathcal{B}_j$ but not $\mathcal{B}_{j+1}$. Now, the description of $(A_n)_n$ under $\mathbb{Q}_{\eta,z}^*$ given by observation (4) in Section \ref{sec:rm} implies that after conditioning on $Z$ and the brother loops of $Z$ at level $(j+1)$, the process \[h_1(-2k+4\log \CR^{-1}(w,A_k))\1_{E_\eta(w,k)} \e^{2k}\CR(w,A_k)^2\] is a martingale for $k\ge j+1$. Hence on the event $\{Z=z\}$ the $j$th term of the sum in \eqref{eqn::brotherloops} is equal to the expected value under $\mathbb Q_{\eta,z}^*$,  conditionally on $\F^*_{A_{k_n}}$, of
	\[ \1_{\Ev_n^3\cap \Ev_n^4} \int_{\mathcal{B}_j\setminus\mathcal{B}_{j+1}} h_1(-2(j+1)+4\log \CR^{-1}(w,A_{j+1})+2\eta)\1_{E_{\eta}(w,j+1)}\e^{2j}\CR(w,A_{j+1})^2dw .\]
	Moreover, this conditional expectation can be bounded above by a constant times
	$$\qa{ \1_{\Ev_n^3\cap \Ev_n^4} \int_{\mathcal{B}_j\setminus \mathcal{B}_{j+1}} h_1(4\log \CR^{-1}(w,\mathcal{B}_j)+2\eta-2j)\e^{2j-2\log\CR^{-1}(w,\mathcal{B}_j)} \, dw \mid \F^*_{A_{k_n}}},$$ because $\CR(w,\D\setminus A_{j})$ is decreasing in $j$, and $h_1$ is bounded on either side by a linear function by \eqref{eqn::h_1_bounds}. 
	
	Finally, note that  on the event $\Ev_n^3\cap \Ev_n^4$,  thanks to K\"{o}ebe's theorem, $2j-2\log \CR^{-1}(w,\mathcal B_j)$ is smaller than $S_j+2c\log(j)+2 \log \CR^{-1}(z,\mathcal B_j)$ for $k_n\le j\le n$, and the area of each $B_j$ is $O(\CR(z,\mathcal{B}_j)^2)$. This means that every term in \eqref{eqn::brotherloops} is $O(\exp(-\sqrt[6]{k_n}/2)n^{4c+1})$, and since $k_n = n^{1/3}$, this therefore  implies (iii).
\end{proofof}
\medskip

We conclude by showing how to extend the proof of Proposition \ref{claim::sh} to treat $\Os$ that are not compactly supported in $\D$. For simplicity, we let $\Os=\D$.

\begin{proof}
	Fix $\eta>0$, and for $\eps>0$ set $\Os_{\eps}:=(1-\eps)\D$. We then know that \begin{equation}\label{eqn::conv_evry_eps}\frac{\sqrt{n}\tilde{K}_n^{\eta,F}(\Os_{\eps})}{\tilde{D}_n^\eta(\Os_{\eps})}\to \theta \E{F(\sqrt{2}R_1)}\end{equation} in $\mathbb{Q}_\eta$-probability as $n\to \infty$, 
	for every $\eps>0$. \footnote{Here, by $\mathbb{Q}_\eta$ we mean the measure defined by \eqref{eqn::q_eta_com} with $\Os=\D$. In fact, for fixed $\eps>0$, Proposition \ref{claim::sh} only tells us that we have the  convergence in probability \eqref{eqn::conv_evry_eps}  when $\mathbb{Q}_\eta$ is defined using $\Os=\Os_{\eps}$. However, since the two probability measures (i.e. when $\mathbb{Q}_\eta$ is defined using $\Os=\Os_{\eps}$ or $\Os=\D$) are absolutely continuous, we can deduce the stated result as well.}
		
	Write 
	\begin{equation}\label{eqn::OequalsDdecomp}
	\frac{\sqrt{n}\tilde{K}_n^{\eta,F}(\D)}{\tilde{D}_n^\eta(\D)}=\frac{\sqrt{n}\tilde{K}_n^{\eta,F}(\Os_{\eps})}{\tilde{D}_n^\eta(\Os_{\eps})}-\frac{\sqrt{n}\tilde{K}_n^{\eta,F}(\Os_{\eps})}{\tilde{D}_n^\eta(\Os_{\eps})}\frac{\tilde{D}_n^\eta(\D\setminus \Os_{\eps})}{\tilde{D}_n^\eta(\D)} + \frac{\sqrt{n}\tilde{K}_n^{\eta,F}(\D\setminus \Os_{\eps})}{\tilde{D}_n^\eta(\D)}	\end{equation} and observe that for any $\delta>0$, by Markov's inequality, 
	\[ \mathbb{Q}_\eta \left(\frac{\tilde{D}_n^\eta(\D\setminus \Os_{\eps})}{\tilde{D}_n^\eta(\D)}>\delta\right)\leq \delta^{-1}\mathbb{E}(\tilde{D}_n^\eta(\D\setminus \Os_{\eps}))=\delta^{-1}\mathbb{E}(\tilde{D}_n^\eta(\D\setminus \Os_{\eps})) \overset{\eps\to 0}{\rightarrow} 0,\]
uniformly in $n$.

	Combining this with \eqref{eqn::conv_evry_eps} and \eqref{eqn::OequalsDdecomp} means that we need only prove, for every $\delta>0$, that
	\begin{equation}\label{eqn::whole_set_unif_prob}\mathbb{Q}_\eta\left(\frac{\sqrt{n}\tilde{K}_n^{\eta,F}(\D\setminus \Os_{\eps})}{\tilde{D}_n^\eta(\D)} >\delta \right)  \overset{\eps\to 0}{\rightarrow} 0\end{equation} uniformly in $n$.
	
	For this, we again use Markov's inequality, and the same strategy that we used to prove the first moment estimate \eqref{eqn::firstmoment}. We write 
	\begin{equation}\label{eqn::whole_set_first_moment}
	\mathbb{Q}_\eta\left(\frac{\sqrt{n}\tilde{K}_n^{\eta,F}(\D\setminus \Os_{\eps})}{\tilde{D}_n^\eta(\D)}\right)=\int_{z\in \D\setminus \Os_{\eps}} \qa{\frac{\sqrt{n}F(S_n(z)/\sqrt{n})}{h_1(S_n(z)+2\eta)}}\mathbb{Q}_\eta^*(dz)
	\end{equation}
	and prove that this expectation converges to $0$ as $\eps\to 0$, uniformly in $n$.
	Note that by \eqref{eqn::persistency}, and the comments following it, the integrand on the right hand side of \eqref{eqn::whole_set_first_moment} is uniformly bounded over $z$ such that $4\log \CR^{-1}(z,\D)\leq n^{1/3}$. 
	Moreover, using that \[\mathbb{Q}_\eta^*[dz]\propto h_1(4\log\CR^{-1}(z,\D)+2\eta)\CR(z,\D)^2 dz,\] and that $F$ and $1/h_1$ are bounded from above, we see that (independently of $\eps$) the integral over the remainder of $\D\setminus \Os_{\eps}$ decays exponentially in $n$. This implies the result, since the area of $\D\setminus \Os_{\eps}$ vanishes as $\eps\to 0$.
	
\end{proof}}

 \subsection* {Acknowledgements.}  We would like to thank R. Rhodes and V. Vargas for elucidating the existing literature in the critical case and N. Berestycki for advice on connecting our measure with the existing Liouville measure in this case. We are also very grateful to W. Werner for presenting us the Neumman-Dirichlet set-up and for inviting E. Powell to visit ETH on two occasions, where a large part of this work was carried out. Finally, we thank the anonymous referee for helpful comments and suggestions. J. Aru and A. Sep\'{u}lveda are supported by the SNF grant \#155922, and are happy to be part of the NCCR Swissmap. E. Powell is supported by a Cambridge Centre for Analysis EPSRC grant EP/H023348/1.
 
\appendix

\section{Renewal functions} \label{sec::appendix}

The material in this appendix comes almost entirely from \cite[Section 2]{AiSh}. \\

	Let $(S_n)_{n\ge 1}$ be a centered random walk under some law $\mathbb{P}$, starting from $S_0=0$, and whose increments have finite variance $\sigma^2$. Then the renewal function $h_1(u)$ is the expected number of times that $(S_n)_n$ hits a strict new minimum before reaching $-u$:

\begin{equation*}
\label{eqn::renew_defn}
h_1(u):= \hat{ \mathbb P}^*\left(\left. \sum_{j=1}^{\infty} \I{\inf_{i\leq j-1} S_i > S_j\geq -u} \, \right| \, Z \right) \ge 1, \ \ u\geq 0.
\end{equation*}
	
    Note that, by the Markov property, we have 
	\begin{equation}\label{eqn::h_1_renewal}
	h_1(u)=\E{h_1(S_n+u)\I{S_n+u\ge 0}}
	\end{equation} for any $n\in \N$. 
	
	By the renewal theorem and our conditions on $(S_n)_n$, it follows that the limit
	\begin{equation}\label{eqn::renewal_thm}
	\lim_{u\to \infty} \frac{h_1(u)}{u} =: c_0 
	\end{equation}
	exists, and $c_0\in (0,\infty)$. Consequently we have 
	\begin{equation}
	\label{eqn::h_1_bounds}
	R'(1+u)\geq h_1(u)\ge R(1+u)
	\end{equation}
	for all $u\geq 0$ and some $R,R'>0$. 
	
	Finally, we also need the following asymptotic estimate from \cite{Kozlov}. For $\theta:=c_0^{-1}\sqrt{\frac{2}{\pi \sigma^2}}$, we have  
	\begin{equation}
	\label{eqn::persistency}\mathbb{P}\left( \min_{1\leq i \leq n} S_i \geq -u \right) \sim \frac{\theta h_1(u)}{\sqrt{n}}
	\end{equation} as $n\to \infty$, for any $u\geq 0$ (see \cite[Formula 12]{Kozlov}). Moreover, it can be shown (\cite{AiJaf}) that this holds uniformly for $u\in [0,b_n]$, whenever $(b_n)_{n\ge 0}$ is a sequence of positive reals such that $\lim_{n\to \infty}b_n n^{-1/2} =0.$

\section{Proof of Theorem \ref{prop::shstrong}} \label{sec::appendixB}

	In this appendix, we show in full detail how Proposition \ref{claim::sh} implies Theorem \ref{prop::shstrong}. 
	
	Fix $\delta>0$,  assume WLOG that $\|F\|_\infty=1$ and pick $u_0>0$ such that $h_1(u)/u\in [c_0-\delta, c_0+\delta]$ for all $u\ge u_0$ (which is possible by \eqref{eqn::renewal_thm}). Then on the event $\cap_z \cap_n E_{\eta-u_0}(z,n)$, we have 
	\[ \left|\frac{\tilde{D}_n^\eta(\Os)-c_0 D_n(\Os)}{D_n(\Os)}\right| \leq \delta+4c_0 \eta \left|\frac{M_n(\Os)}{D_n(\Os)}\right|. \]
	If we assume in addition that we are on the event 
	\[ \left\{\left|\frac{M_n(\Os)}{D_n(\Os)}\right|\leq \frac{\delta}{c_0\eta}\right\} \cap  \left\{\left|c_0\frac{\sqrt{n}\tilde{K}^{F,\eta}_n(\Os)}{\tilde{D}^\eta_n(\Os)}-\frac{1}{\sqrt{\pi}} \E{F(\sqrt{2}R_1)}\right|\leq \delta\right\}=: A_{n,\eta}\cap A'_{n,\eta} \]
	then (as long as $\delta$ is small enough) we have 
	\begin{align*} &\left|\frac{\sqrt{n}K_n^F(\Os)}{D_n^F(\Os)}-\frac{1}{\sqrt{\pi}} \E{F(\sqrt{2}R_1)}\right|  \\   &\hspace{0.1\textwidth}\leq  \left|\frac{\sqrt{n}\tilde{K}^{F,\eta}_n(\Os)}{\tilde{D}^\eta_n(\Os)} \right|\left| \frac{\tilde{D}_n^\eta(\Os)-c_0 D_n(\Os)}{D_n(\Os)}\right| +  \left|c_0\frac{\sqrt{n}\tilde{K}^{F,\eta}_n(\Os)}{\tilde{D}^\eta_n(\Os)}-\frac{1}{\sqrt{\pi}} \E{F(\sqrt{2}R_1)} \right|,\end{align*}
	which by definition is smaller than $3\delta$. Therefore, we can bound
	\[ \mathbb{P}\left(\left|\sqrt{n}\frac{K_n^F(\Os)}{D_n^F(\Os)}- \frac{1}{\sqrt \pi} \E{F(\sqrt{2}R_1)}\right|>3\delta \right) \le \mathbb{P}(\cup_z \cup_n E_{\eta-u_0}(z,n)) + \mathbb{P}((A_{n,\eta})^c)+ \mathbb{P}((A'_{n,\eta})^c),\] which, by the definition of $\mathbb{Q}^\eta$, is for any $K>0$ less than or equal to 
	\[\mathbb{P}(\cup_z \cup_n E_{\eta-u_0}(z,n)) + \mathbb{P}((A_{n,\eta})^c)+\mathbb{P}\left (\tilde{D}_n^\eta(\Os) \leq \frac{1}{K}\right )+\frac{K}{\tilde{D}_0^\eta(\Os)}\mathbb{Q}^\eta((A'_{n,\eta})^c).\] 
	Now take $\eps>0$. By \eqref{eqn::unifmin}, Proposition \ref{propn::derivative} and Proposition \ref{lem::criticaltozero}, we can pick $u_0,\eta, K$ and $N_0 \in \N$ such that the first three terms above are each less than $\eps/4$ for all $n\ge N_0$. Then, using Proposition \ref{claim::sh}, we can choose a further $N_0'\geq N_0$ such that the final term is also less than $\eps/4$ for $n\geq N_0'$. Since $\eps$ and $\delta$ were arbitrary, this concludes the proof of the theorem.

\bibliographystyle{alpha}
\bibliography{bibliography_gmc_cascades}

\end{document}